\documentclass[reqno]{amsart}
\usepackage{amsfonts}

\usepackage{amsmath,amssymb,amsthm,amsfonts}
\usepackage{bbm}
\usepackage{hyperref}

\usepackage{color}

\newtheorem{lemma}{Lemma}[section]
\newtheorem{theorem}{Theorem}[section]

\newtheorem{proposition}{Proposition}[section]

\numberwithin{equation}{section}

\newcommand{\dis}{\displaystyle}

\newcommand{\R}{\mathbb{R}}

\newcommand{\al}{\alpha}
\newcommand{\be}{\beta}
\newcommand{\ga}{\gamma}

\newcommand{\la}{\lambda}

\newcommand{\si}{\sigma}
\newcommand{\pa}{\partial}
\newcommand{\ka}{\kappa}
\newcommand{\eps}{\epsilon}

\newcommand{\lag}{\langle}
\newcommand{\rag}{\rangle}

\begin{document}

\title[Stability of rarefaction waves of Navier-Stokes-Poisson system]{Stability of rarefaction waves of the Navier-Stokes-Poisson system}

\author[R.-J. Duan]{Renjun Duan}
\address[RJD]{Department of Mathematics, The Chinese University of Hong Kong,
Shatin, Hong Kong, P.R.~China}
\email{rjduan@math.cuhk.edu.hk}

\author[S.-Q. Liu]{Shuangqian Liu}
\address[SQL]{Department of Mathematics, Jinan University, Guangzhou 510632, P.R.~China\\
and Department of Mathematics, The Chinese University of Hong Kong,
Shatin, Hong Kong, P.R.~China}
\email{tsqliu@jnu.edu.cn}


\maketitle
\begin{abstract}
In the paper we are concerned with the large time behavior of solutions to the one-dimensional Navier-Stokes-Poisson system in the case when the potential function of the self-consistent electric field may take distinct constant states at $x=\pm\infty$. Precisely, it is shown that if initial data are close to a constant state with asymptotic values at far fields chosen such that the Riemann problem on the corresponding quasineutral Euler system admits a  rarefaction wave whose strength is not necessarily small, then the solution exists for all time and tends to the rarefaction wave as $t\to+ \infty$. The construction of the nontrivial large-time profile of the potential basing on the quasineutral  assumption plays a key role in the  stability analysis. The proof is based on the energy method by taking into account the effect of the self-consistent electric field on the viscous compressible fluid.
\end{abstract}


\tableofcontents
\thispagestyle{empty}
\section{Introduction}
The two-fluid Navier-Stokes-Poisson (denoted as NSP in the sequel) system is often used to describe the transport of charged particles under the influence of the self-consistent electrostatic potential force arising from the study of the collisional dusty plasma, cf.~\cite{GSK,KG}.
In one-dimensional space, it takes the form of
\begin{equation}\label{NSP2}
\left\{\begin{array}{l}
\pa_tn_i+\pa_x(n_i u_i)=0,\\[3mm]
m_in_i(\pa_t u_i+u_i\pa_x u_i)+T_i\pa_x n_i-n_i\pa_x\phi=\mu_i \pa_x^2u_i,\\[3mm]
\pa_tn_e+\pa_x(n_e u_e)=0,\\[2mm]
m_en_e(\pa_t u_e+u_e\pa_x u_e)+T_e\pa_x n_e+n_e\pa_x\phi=\mu_e \pa_x^2u_e,\\[3mm]
\pa_x^2\phi=n_i-n_e,\quad t>0,\ x\in \R.
\end{array}
\right.
\end{equation}
Initial data are given by
\begin{equation}\label{I.D.2}
[n_\al, u_\al](0,x)=[n_{\al0}(x),u_{\al0}(x)],\quad \al=i,e,\  \ x\in\R,
\end{equation}
with
\begin{equation}
\label{I.D.2-c}
\lim\limits_{x\rightarrow\pm\infty}[n_{\al0},u_{\al0}](x)=[n_{\pm},u_{\pm}],\quad \al=i,e.
\end{equation}
The boundary values of $\phi$ at infinity are set by
\begin{equation}
\label{bd2}
\lim\limits_{x\rightarrow\pm\infty}\phi(t,x)=\phi_\pm,\quad t\geq 0.
\end{equation}
Here, $n_\al=n_\al(t,x)>0$ and $u_\al=u_\al(t,x)$ are the density and velocity of ions $(\al=i)$ and electrons $(\al=e)$, respectively, and  $\phi=\phi(t,x)$ is the self-consistent potential. The  positive constants $m_\al>0$, $T_\al>0$ and   $\mu_\al>0$ denote respectively the mass, the absolute temperature and the viscosity coefficient of $\al$-fluid. Constant states $[n_\pm,u_\pm,\phi_\pm]$ at infinity can be distinct. Particularly, we allow for the appearance of nonzero difference of potentials at $x=\pm\infty$, i.e.~$\phi_+-\phi_-\neq 0$.

It is interesting to study the large-time behavior of solutions to the Cauchy problem on the complex NSP system \eqref{NSP2}, \eqref{I.D.2}, \eqref{I.D.2-c}, \eqref{bd2} in the case when $[n_-,u_-,\phi_-]\neq [n_+,u_+,\phi_+]$. In the paper we are concerned with the time-asymptotic stability of the rarefaction wave (cf.~\cite{Da,S}) constructed by the corresponding quasineutral Euler system
\begin{equation}\label{Euler2}
\left\{\begin{array}{l}
\dis \pa_tn+\pa_x(n u)=0,\\[3mm]
\dis n(\pa_t u+u\pa_x u)+\frac{T_i+T_e}{m_i+m_e}\pa_x n=0,
\end{array}
\right.
\end{equation}
with the potential function $\phi$ in large time determined by
\begin{equation}
\label{phi2}
\phi=\frac{T_im_e-T_em_i}{m_i+m_e} \ln n.
\end{equation}
System \eqref{Euler2} can be formally obtained from  \eqref{NSP2}  by letting $n_i=n_e=n$, $u_i=u_e=u$, taking the sum of two momentum equations and neglecting viscosity terms, and the relationship \eqref{phi2} can be deduced by further taking the difference of  two momentum equations, neglecting viscosity terms, and using the quasineutral momentum equation in  \eqref{Euler2}. Therefore, we need to postulate the following compatibility condition on data $[n_\pm,\phi_\pm]$ at infinity
\begin{equation}
\label{cc2}
\phi_\pm=\frac{T_im_e-T_em_i}{m_i+m_e} \ln n_\pm.
\end{equation}
Notice from \eqref{cc2} that if $T_i/m_i\neq T_e/m_e$ then the distinct $n_\pm$ can yield the distinct $\phi_\pm$. Precisely, we expect to show
\begin{equation*}
n_\al (t,x)\to n^R(x/t),\quad u_\al (t,x)\to u^R(x/t),\quad \al=i,e,
\end{equation*}
and
\begin{equation*}
\phi(t,x)\to  \phi^R (x/t):=\frac{T_im_e-T_em_i}{m_i+m_e} \ln n^R(x/t),
\end{equation*}
uniformly for $x\in \R$ as $t$ goes to infinity, provided that initial data $[n_{\al0}(x),u_{\al0}(x)]$ approach $[n_\pm,u_\pm]$ as $x\to \pm \infty$  in a suitably close way, where  $[n^R(x/t),u^R(x/t)]$  is the rarefaction-wave solution to the corresponding Riemann problem on the quasineutral Euler system  \eqref{Euler2}.


In fact, there have been a huge number of papers in the literature to study of the stability of wave patterns, namely, shock wave, rarefaction wave, contact discontinuity and their compositions, in the context of gas dynamical equations and related kinetic equations. Among them, we only mention \cite{Go,HLM,HXY,LYYZ,LY,MN86,MN-S,Y} and reference therein.  Moreover, we would also point out some previous works only related to the current work concerning the stability of rarefaction waves. The problem was firstly studied by Matsumura-Nishihara \cite{MN86,MN92} for the one-dimensional  compressible Navier-Stokes equations in the isentropic case. It was later extent by Liu-Xin \cite{LX} to the heat-conductive case. Since then, there have been extensive studies in connection with considerations of different aspects, such as appearance of boundaries \cite{KZ,M}, dependence of viscosity on density \cite{JQZ}, other kinds of relative models \cite{KT, LYYZ}, and so on.

Another interesting model is the Navier-Stokes system coupled with the Poisson equation through the self-consistent force arising either from modelings of self-gravitational viscous gaseous stars (cf.~\cite{Chan}) or from the simulation of the motion of charged particles in semiconductor devices (cf.~\cite{MRC}). The coupling system at the fluid level can be also justified by taking the hydrodynamical limit of the Vlasov-type Boltzmann equation by the Chapman-Enskog expansion, cf.~\cite{CC,Gr,G,GJ}.   In recent years, the study of the NSP system has attracted a lot of attentions from many people. In what follows we would also mention some of them related to our interest. In the case when the background profile is vacuum, the existence of nontrivial stationary solutions with compact support and their dynamical stability related to a free-boundary value problem for the three-dimensional NSP system were discussed in  Ducomet \cite{Du}. Global existence of weak solutions to the Cauchy problem with large initial data was proved by Donatelli \cite{D} and the quasineutral limit in such framework was studied in  Donatelli-Marcati \cite{DM} by using some dispersive estimates of Strichartz type. We remark that some nonexistence result of global weak solutions was also very recently obtained in Chae \cite{Chae}.    Large-time behavior of the spherically symmetric NSP  system with degenerate viscosity coefficients and with vacuum in three dimensions was obtained in Zhang-Fang \cite{ZF}. The linear and nonlinear
dynamical instability for the Lane-Emden solutions in the framework of the
NSP system in three dimensions  was investigated in Jang-Tice \cite{JT} under some condition on the adiabatic exponent. Tan-Yang-Zhao-Zou \cite{TYZZ} established the global strong solution to the one-dimensional non-isentropic NSP system with large data for density-dependent viscosity.  In the case when the background profile is strictly positive, the global existence and convergence rates for the three-dimensional NSP system around a non-vacuum constant state were studied by Li-Matsumura-Zhang \cite{LMZ} through the construction of Green's functions. Interested readers may refer to the survey paper Hsiao-Li \cite{HL}  and reference therein for the perturbation theory related to the NSP system; see also \cite{D-NSM} for the study of the more complicated Navier-Stokes-Maxwell system.

However, to the best of our knowledge, even though there are many studies on the NSP system mentioned above, so far there are few mathematical results to clarify the nonlinear stability of wave patterns. One of the main mathematical difficulties comes from the effect of the self-consistent force on the viscous compressible fluid, since the force generally may not be expected to be $L^2$ integrable in space and time, cf.~\cite{LMZ}. Recently, Duan-Yang \cite{DY} proved the stability of rarefaction wave and boundary layer for outflow problem on the two-fluid NSP  equations. The convergence rate of corresponding solutions toward the stationary solution was obtained in Zhou-Li \cite{ZL}.  We point out that due to the techniques of the proof, it was assumed in \cite{DY} that all physical parameters in the model must be unit, particularly $m_i=m_e$ and $T_i=T_e$, which is obviously unrealistic since ions and electrons generally have different masses and temperatures.  One key point used in \cite{DY} is that the large-time {behavior} of the electric potential is trivial and hence the two fluids indeed have the same asymptotic profiles which are constructed from the Navier-Stokes equations without any force instead of the quasineutral system that we will make use of in the paper.

Motivated by \cite{GHR,GP} for the study of the inviscid Euler-Poisson equations, we will be firstly concerned with the motion of the one viscous fluid of ions under the Boltzmann relation.  It will be seen from Theorem \ref{main.res.} that the rarefaction wave of the quasineutral Euler system is stable under small perturbation, and particularly, the potential $\phi(t,x)$ has the nontrivial large-time behavior. Compared to the classical Navier-Stokes system without any force, the main difficulty in the proof for the NSP system is to treat the estimates on those terms related to the potential function $\phi$ as mentioned above. Since the large-time {behavior} of $\phi$ has a slow time-decay rate and the strength of rarefaction waves is not necessarily small, it is quite nontrivial to estimate the coupling term $-n \pa_x\phi$. The key point to overcome the difficulty is to use the good dissipative property from  the Poisson equation by expanding $e^{-\phi}$ around the asymptotic profile up to the third-order. In the two-fluid case, the situation is more complicated since the dissipation of the system  becomes much weaker than that in the case of one-fluid ions. We find that the trouble term turns out to be controlled by taking the difference of two momentum equations with different weights so as to balance the different masses $m_\al$ $(\al=i,e)$ of fluids. Finally we point out that it could be expected to use the developed techniques in this paper to pursuit the proof of the stability of smooth traveling waves.

The rest of the paper is organized as follows. In Section 2, we present the nonlinear stability of rarefaction waves for the NSP system in the relatively simple case when only the ions flow is taken into account under the Boltzmann relation. Precisely, we first construct a smooth approximation of the rarefaction wave, reformulate the equations around the smooth asymptotic profile, and then show the global a priori estimates on the solutions so as to obtain the main result. The stability of rarefaction waves for the more physical two-fluid model \eqref{NSP2} is dealt with in a similar way in Section 3. The symmetric structure in the two-fluid case plays a key in the proof. For convenience of readers, the linear dissipative structure of the system for the model considered in Section 2 is also analyzed in the appendix in order to understand the energy estimate around the nontrivial profile.

\medskip
\noindent {\it Notations.} Throughout this paper,  $C$ denotes some generic positive (generally large) constant and $\la$ denotes some generic positive (generally small) constant, where both $C$ and $\la$ may take different values in different places. $D\lesssim E$ means that  there is a generic constant $C>0$
such that $D\leq CE$. $D\sim E$
means $D\lesssim E$ and $E\lesssim D$. $\|\cdot\|_{L^p}$ $(1\leq p\leq+\infty)$
stands for the $L_x^p-$norm. Sometimes, for convenience, we use $\|\cdot\|$ to denote $L_x^2-$norm, and use $(\cdot,\cdot)$ to denote the inner product in $L^2_x$.
We also use $H^{k}$ $(k\geq0)$ to denote the usual Sobolev space with respect to $x$ variable.

\section{One-fluid case}

In this section we study the stability of rarefaction waves for the NSP system in the case of single ions flow under the Boltzmann relation. The main motivation for taking into account the ions flow first is that the structure of the coupling system is relatively simpler than that in the two-fluid case given in \eqref{NSP2} so that the analysis for the stability of rarefaction waves in the one-fluid case can shed light on the one in the more complex two-fluid case. Of course, the model for the single viscous flow for ions also has its own interest, particularly in the study of smooth traveling waves, cf.~\cite{GHR,KG}.

After a suitable normalization, the simplified system in the case of single ions flow reads
\begin{equation}\label{NSP}
\left\{\begin{array}{l}
\pa_tn+\pa_x(n u)=0,\\[3mm]
{n(\pa_t u+u\pa_x u)+A\pa_xn-n \pa_x\phi=\pa_x^2u},\\[3mm]
\pa_x^2\phi=n-e^{-\phi}, \quad t>0,\ x\in \R,
\end{array}
\right.
\end{equation}
with
\begin{equation}
\label{con.phi}
\lim\limits_{x\rightarrow\pm\infty}\phi(t,x)= \phi_\pm, \quad t\geq 0.
\end{equation}
Initial data are given by
\begin{eqnarray}\label{I.D.}
\left\{\begin{array}{rll}
&&{[n, u](0,x)=[n_0,u_0](x)},\ \ x\in\R,\\[2mm]
&&{\lim\limits_{x\rightarrow\pm\infty}[n_0,u_0](x)=[n_{\pm},u_{\pm}]}.
\end{array}
\right.
\end{eqnarray}
Here $n=n(t,x)>0,$ and $u=u(t,x)$ are the density and velocity of ions, and $\phi=\phi(t,x)$ is the potential of the self-consistent electric field. The  positive constant $A>0$ is the absolute temperature of ions.
Note from the Poisson equation in \eqref{NSP} that the Boltzmann relation $n_e=e^{-\phi}$ has been assumed, which is physically due to the fact that lighter electrons get close to the equilibrium state in much faster rate than heavier ions in plasma, cf.~\cite{Ch,HK2}. We remark that  $n_e=e^{-\phi}$ can be formally derived from the fourth equation of the two-fluid model \eqref{NSP2} by letting $m_e=0=\mu_e$ {and $T_e=1$}. Throughout this section, we also assume that  the quasineutral condition
\begin{equation}
\label{con.n}
n_\pm=e^{-\phi_\pm}
\end{equation}
at $x=\pm\infty$ holds true. Notice that $n_\pm$ and hence $\phi_\pm$ can be distinct.

\subsection{Approximate rarefaction waves}

In general, whenever $\phi$ is a nontrivial potential in large time and its second-order derivative $\pa_x^2\phi$ decays in time faster than other low-order terms, one may expect that  the NSP system \eqref{NSP} tends to  the following quasineutral Euler equation:
\begin{eqnarray}\label{ME.}
\left\{\begin{array}{rll}
&&\pa_t\overline{n}+\pa_x(\overline{n}~\overline{u})=0,\\[2mm]
&&{\overline{n}(\pa_t \overline{u}+\pa_x \overline{u})+A\pa_x\overline{n}-\overline{n}\pa_x\overline{\phi}=0},\\[2mm]
&&\overline{\phi}=-\ln\overline{n}.
\end{array}
\right.
\end{eqnarray}
Specifically, the large time behavior of solutions to the Cauchy problem \eqref{NSP}, \eqref{con.phi}, \eqref{I.D.} and \eqref{con.n} of  the NSP system is expected to be determined by the Riemann problem for the above quasineutral Euler equation with initial data
\begin{eqnarray*}
\begin{array}{rll}
[\overline{n}, \overline{u}](0,x)&&=[\overline{n}_0,\overline{u}_0](x)=\left\{\begin{array}{rll}[n_-,u_-],&\ \ x>0,\\[3mm]
[n_+,u_+],&\ \ x<0,
\end{array}
\right.\\[3mm]
\overline{\phi}(0,x)&&=\overline{\phi}_0=-\ln \overline{n}_0.
\end{array}
\end{eqnarray*}
It is easy to see (cf.~\cite{Da,S}) that \eqref{ME.}
have two characteristics
\begin{equation}
\label{1cha}
\la_1=u-c,\ \ \la_2=u+c,\quad \text{with}\ c=\sqrt{A+1},
\end{equation}
which are genuinely nonlinear and can give rise to the rarefaction waves:
\begin{equation*}
R_1(n_-,u_-)\equiv\bigg\{{[n,u]}\in\R_+\times\R\ \Big|\ u+c\ln n=u_-+c\ln n_-,
n<n_-, u>u_-\bigg\},
\end{equation*}
or
\begin{equation}\label{def.r2}
R_2(n_-,u_-)\equiv\bigg\{{[n,u]}\in\R_+\times\R\ \Big|\ u-c\ln n=u_--c\ln n_-,n>n_-,\ u>u_-\bigg\}.
\end{equation}

In what follows we construct a smooth approximation of the rarefaction waves corresponding to the solution to the Riemann problem \eqref{ME.} on the quasineutral Euler system. As in \cite{MN86},
consider the Riemann problem on the Burgers' equation
\begin{eqnarray}\label{Re.p}
\left\{\begin{array}{rll}
\begin{split}
w_t+ww_x&=0,\\
w(x,0)&=w_0=\left\{\begin{array}{rll}
w_-,\ x&<0,\\[3mm]
w_+,\ x&>0,\end{array}\right.
\end{split}
\end{array}\right.
\end{eqnarray}
for $w_-<w_+$. It is well-known that \eqref{Re.p} admits a continuous weak solution $w^{R}(x/t)$ connecting $w_-$ and
$w_+$, taking the form of
\begin{equation*}
w^{R}(x/t)=\left\{\begin{array}{ll}
w_-&,\ x\leq w_-t,\\[3mm]
\dis \frac{x}{t}&,\ w_-t<x<w_+t,\\[3mm]
w_+&,\ w_+t\leq x.
\end{array}\right.
\end{equation*}
Let $[n^{R},u^{R}](x/t)$ be defined as
\begin{equation}\label{app.sol.}
w^R(x/t)=u^{R}+c,\ \ u^R-c\ln n^R=u_--c\ln n_-,
\end{equation}
with $w_-=u_-+c$ and $w_+=u_++c$.
Then by a direct calculation, $[n^{R},u^{R}](x/t)$ satisfies the following Riemann problem on the Euler equations
\begin{eqnarray*}
\left\{\begin{array}{rll}
&&\pa_t\overline{n}+\pa_x(\overline{n}~\overline{u})=0,\\[2mm]
&&{\overline{n}\pa_t (\overline{u}+\pa_x \overline{u})+c^2\pa_x\overline{n}}=0,\\[2mm]
&&[\overline{n},\overline{u}](x,0)=[\overline{n}_0,\overline{u}_0]=\left\{
\begin{array}{ll}
[n_-,u_-]&,\ \ x<0,\\[3mm]
 {[n_+,u_+]} &,\ \ x>0.
 \end{array}\right.
\end{array}
\right.
\end{eqnarray*}
In fact, the smooth Euler $i$-th $(i=1,2)$ rarefaction waves can be constructed along any given $R_i$ curve when the $i$-th characteristic satisfies the inviscid Burgers
equation
$$
\pa_t\la_i+\la_i\pa_x\la_i=0,
$$
with increasing data. Here we refer to the construction introduced in \cite{MN86}
for initial data whose spatial derivative is  proportional to the
parameter $\eps>0$, i.e., $w^R(x/t)$ is approximated by a smooth function $\overline{w}(x,t)$ satisfying
\begin{equation}\label{cl.Re.cons.}
\left\{\begin{array}{rll}
\begin{split}
\pa_tw+w\pa_xw&=0,\\
w(0,x)&=w_0(x)=\frac{1}{2}(w_{+}+w_{-})+\frac{1}{2}(w_{+}-w_{-})\tanh (\eps x).
\end{split}
\end{array}\right.
\end{equation}
Then one has

\begin{lemma}\label{cl.Re.Re.}
Let $\delta=w_+-w_->0$ be the wave strength of the $2$-rarefaction wave. Then the problem \eqref{cl.Re.cons.} has a unique smooth solution
$\overline{w}(t,x)$ which satisfies the following properties:

\noindent$(i)$ $w_-<\overline{w}(t,x)<w_+$, $\pa_x\overline{w}>0$ for $x\in\R$ and $t\geq0$.

\noindent$(ii)$ For any $1\leq p\leq+\infty$, there exists a constant $C_p$ such that for $t>0$,
\begin{eqnarray*}
&&\|\pa_x\overline{w}\|_{L^p}\leq C_p\min\{\delta\eps^{1-1/p}, \delta^{1/p}t^{-1+1/p}\},\\
&&\|\pa^2_x\overline{w}\|_{L^p}\leq C_p\min\{\delta\eps^{2-1/p}, \eps^{1-1/p}t^{-1}\},\\
&&\|\pa^3_x\overline{w}\|_{L^p}\leq C_p\min\{\delta\eps^{3-1/p}, \eps^{2-1/p}t^{-1}\}.
\end{eqnarray*}

\noindent$(iii)$ $\lim\limits_{t\rightarrow+\infty}\sup\limits_{x\in\R}\left|\overline{w}(t,x)-w^R(x/t)\right|=0$.

\end{lemma}
We now define the approximate solution of $[n^R,u^R]$ constructed by \eqref{app.sol.} and \eqref{cl.Re.cons.} as
$[n^{r}, u^{r}]=[n^{r}, u^{r}](t,x)$ in terms of
\begin{eqnarray*}
\left\{\begin{array}{rll}
&&\overline{w}(t+1,x)=u^{r}(t,x)+c, \\[2mm]
&&u^{r}(t,x)-c\ln n^{r}(t,x)=u_--c\ln n_-,
\end{array}
\right.
\end{eqnarray*}
together with $w_-=u_-+c$ and $w_+=u_++c$. We also point out  that $[n^{r}, u^{r}]$ satisfies
\begin{eqnarray}\label{Re.ME2.}
\left\{\begin{array}{rll}
&&\pa_t\overline{n}+\pa_x(\overline{n}~\overline{u})=0,\\[2mm]
&&{\overline{n}\pa_t (\overline{u}+ \pa_x\overline{u})+A\pa_x\overline{n}-\overline{n}\pa_x\overline{\phi}=0},
\\[2mm]
&&\overline{\phi}=-\ln\overline{n},\\[2mm]
&&[\overline{n},\overline{u}](0,\pm\infty)=[n_\pm, u_\pm].
\end{array}
\right.
\end{eqnarray}
With Lemma \ref{cl.Re.Re.} in hand, one has the following result concerning $[n^{r}, u^{r}]$.

\begin{lemma}\label{cl.Re.Re2.}
The approximate smooth solution $[n^{r}, u^{r}]$  given by \eqref{app.sol.} and \eqref{cl.Re.cons.}
satisfies the following properties:

\noindent$(i)$ $\pa_xn^{r}>0$, $\pa_xu^{r}>0$ and $n_-<n^{r}(t,x)<n_+$,
$u_-<u^{r}(t,x)<u_+$ for $x\in\R$ and
$t\geq0$.

\noindent$(ii)$ For any $1\leq p\leq+\infty$, there exists a constant $C_p$ such that for all $t>0$,
\begin{eqnarray*}
&&\|\pa_x[n^{r}, u^{r}]\|_{L^p}\leq C_p\min\{\delta_r\eps^{1-1/p}, \delta_r^{1/p}t^{-1+1/p}\},\\
&&\|\pa^2_x[n^{r}, u^{r}]\|_{L^p}\leq C_p\min\{\delta_r\eps^{2-1/p}, \eps^{1-1/p}t^{-1}\},\\
&&\|\pa^3_x[n^{r}, u^{r}]\|_{L^p}\leq C_p\min\{\delta_r\eps^{3-1/p}, \eps^{2-1/p}t^{-1}\},
\end{eqnarray*}
where $\delta_r=|n_+-n_-|+|u_+-u_-|$.

\noindent$(iii)$ $\lim\limits_{t\rightarrow+\infty}\sup\limits_{x\in\R}\left|[n^{r}, u^{r}](t,x)-[n^R, u^R](t,x)\right|=0$.

\end{lemma}

\subsection{Stability of the rarefaction wave}
In this subsection, we study the stability of the approximate rarefaction wave $[n^{r}, u^{r}]$
for the Cauchy problem \eqref{NSP}, \eqref{con.phi}, \eqref{I.D.} and \eqref{con.n}. For this purpose, let us define the perturbation
$$
[\widetilde{n}, \widetilde{u}, \widetilde{\phi}](t,x)=[n-n^{r},u-u^{r}, \phi-\phi^{r}](t,x),
$$
where $\phi^{r}=-\ln n^{r}.$
Then $[\widetilde{n},\widetilde{u}, \widetilde{\phi}]$ satisfies
\begin{eqnarray}
&&\pa_t\widetilde{n}+\pa_x(n u)-\pa_x(n^ru^r)=0,\label{trho}\\
&&\pa_t \widetilde{u}+u\pa_x u-u^r\pa_x u^r+A\left(\frac{\pa_xn}{n}-\frac{\pa_xn^r}{n^r}\right)
-\pa_x\widetilde{\phi}=\frac{\pa_x^2u}{n},\label{tu}\\
&&\pa_x^2\widetilde{\phi}=\widetilde{n}+n^r\left(1-e^{-\widetilde{\phi}}\right)-\pa_x^2\phi^r.\label{tphy}
\end{eqnarray}
Initial data for $[\widetilde{n},\widetilde{u}]$ are given by
\begin{eqnarray}
&&[\widetilde{n},\widetilde{u}](0,x)=[\widetilde{n}_0,\widetilde{u}_0](x)=[n_0(x)-n^r(0,x),u_0(x)-u^r(0,x)].\label{tid}
\end{eqnarray}
Here it has been assumed that for any given $t\geq 0$, $\widetilde{\phi}(t,x)$ is determined in terms of the elliptic equation \eqref{tphy} under the boundary condition that $\widetilde{\phi}(t,x)$ tends to zero as $x$ goes to $\pm\infty$.

\begin{theorem}\label{main.res.}
Let ${[n_+, u_+]}\in R_2(n_-, u_-)$, where $R_2(n_-, u_-)$ is given in \eqref{def.r2} with $n_->0$ and $c=\sqrt{A+1}$. Assume $\phi_\pm=-\ln n_\pm$.
There are constants $\eps_0>0$, $0<\si<1$, and $C_0>0$ such that if
$$
\left\|[\widetilde{n}_0, \widetilde{u}_0]\right\|^2_{H^1}+\eps\leq \eps_0^2,
$$
where $\eps>0$ is the parameter appearing in \eqref{cl.Re.cons.}, then the Cauchy problem of
the NSP system \eqref{trho}, \eqref{tu}, \eqref{tphy} and \eqref{tid}
admits a unique global solution $[\widetilde{n},\widetilde{u},\widetilde{\phi}](t,x)$ satisfying
\begin{multline}\label{total.eng.}
\sup\limits_{t\geq0}\left\|\left[\widetilde{n}, \widetilde{u}, \widetilde{\phi}\right](t)\right\|^2_{H^1}
+\int_0^{+\infty}\left\|\sqrt{\pa_xu^r}\left[ \widetilde{u}, \pa_x\widetilde{n},\pa_x\widetilde{u}\right](t)\right\|^2dt
\\+\int_0^{+\infty}\left\|\left[\pa_x\widetilde{n}, \pa_x\widetilde{u}, \pa_x\widetilde{\phi}, \pa^2_x\widetilde{\phi}, \pa^2_x\widetilde{u}\right](t)\right\|^2dt
\leq C_0\eps_0^{\frac{2\sigma}{1+\sigma}}.
\end{multline}
Moreover, the solution to the original Cauchy problem \eqref{NSP}, \eqref{con.phi}, \eqref{I.D.} tends time-asymptotically to the rarefaction in the sense that
\begin{equation}\label{asp}
\lim\limits_{t\rightarrow+\infty}\sup\limits_{x\in\R}\left|\left[n(t,x)-n^{R}\left(\frac{x}{t}\right), u(t,x)- u^{R}\left(\frac{x}{t}\right), \phi(t,x)+\ln n^{R} \left(\frac{x}{t}\right)\right]\right|=0.
\end{equation}

\end{theorem}

The local existence of the solution $[\widetilde{n},\widetilde{u},\widetilde{\phi}]$  of the
NSP system \eqref{trho}, \eqref{tu}, \eqref{tphy} and \eqref{tid} can be proved by the standard iteration method and its proof is omitted for brevity.
As to the proof of Theorem \ref{main.res.}, it suffices to show the following a priori estimate.

\begin{proposition}\label{a.pri.}
Assume that all the conditions in Theorem \ref{main.res.} hold. There are constants $\eps_0>0$, $0<\si<1$, and $C_0>0$ such that if the smooth solution $[\widetilde{n},\widetilde{u},\widetilde{\phi}]$ to the Cauchy problem
\eqref{trho}, \eqref{tu}, \eqref{tphy} and \eqref{tid} on $0\leq t\leq T$ for $T>0$ satisfies
\begin{equation}\label{pri.ass.}
\sup\limits_{0\leq t\leq T}\left\|\left[\widetilde{n}, \widetilde{u}, \widetilde{\phi}\right](t)\right\|^2_{H^1}+\eps\leq \eps_0^2,
\end{equation}
then it holds that
\begin{multline}\label{total.eng1.}
\sup\limits_{0\leq t\leq T}\left\|\left[\widetilde{n}, \widetilde{u}, \widetilde{\phi}\right](t)\right\|^2_{H^1}
+\int_0^{T}\left\|\sqrt{\pa_xu^r}\left[ \widetilde{u}, \pa_x\widetilde{n},\pa_x\widetilde{u}\right](t)\right\|^2dt
\\+\int_0^{T}\left\|\left[\pa_x\widetilde{n}, \pa_x\widetilde{u}, \pa_x\widetilde{\phi}, \pa^2_x\widetilde{\phi}, \pa^2_x\widetilde{u}\right](t)\right\|^2dt
\leq C_0\left(\left\|[\widetilde{n}, \widetilde{u}](0)\right\|^2_{H^1}+\eps^{\frac{\sigma}{1+\sigma}}\right).
\end{multline}
\end{proposition}

We remark that the strength of the rarefaction wave $|n_+-n_-|+|u_+-u_-|$ is not necessarily small, and the similar result could hold for the non-isothermal case, for instance, for a general pressure function depending on the density. The proof of Proposition \ref{a.pri.} is based on the techniques developed in \cite{MN92} for the classical Navier-Stokes equations. We here have to additionally take into account the effect of the self-consistent force on the system. The most technical part in the proof occurs to the estimate on the inner product $(-\pa_x\widetilde{\phi},n \widetilde{u})$ for which one has to make use of the Poisson equation by expending $e^{-\widetilde{\phi}}$ up to the third-order {and obtain cancelations by seeking the hidden relation among the lower order terms}.

\subsection{A priori estimate}

\begin{proof}[Proof of Proposition \ref{a.pri.}]
We divide it by three steps as follows.

\medskip
\noindent{\bf Step 1.} {\it Zero-order energy estimates.}

Set
\begin{equation}
\label{def.psi}
\psi(n,n^r)=A\int_{n^r}^n\frac{s-n^r}{s^2}ds.
\end{equation}
Note that $\psi(n,n^r)$ is equivalent to $|n-n^r|^2$ for $|n-n^r|\leq C$.

Multiplying  the first equation of \eqref{NSP} and \eqref{tu} by $\psi(n,n^r)$ and $n\widetilde{u}$ respectively,
 integrating the resulting equations with  respect to $x$ over $\R$ and taking the summation, one obtains
\begin{equation}\label{ztu.ip}
\frac{1}{2}\frac{d}{dt} (\widetilde{u}, n\widetilde{u})+\frac{d}{dt} (n, \psi)\underbrace{-\frac{1}{2}(\widetilde{u}^2, n_t)+(u\pa_x u-u^r\pa_x u^r,n\widetilde{u})}_{I_1}
-(\pa_x\widetilde{\phi}, n\widetilde{u})=(\pa_x^2\widetilde{u},\widetilde{u})+(\pa_x^2u^r,\widetilde{u}),
\end{equation}
where we have used the identity
$$
-(n,\pa_t\psi)+(\pa_x(n u),\psi)+A\left(\frac{\pa_xn}{n}-\frac{\pa_xn^r}{n^r},n\widetilde{u}\right)=0,
$$
which can be justified by replacing $\psi(n,n^r)$ by \eqref{def.psi} and using equations of $\pa_tn$ and $\pa_t n^r$ as well as integration by part.

The left-hand terms of  \eqref{ztu.ip} are estimated as follows.  For $I_1$, from the first equation of \eqref{NSP}, it holds that
\begin{equation}\label{ztu.ip1}
\begin{split}
I_1=\frac{1}{2}(\widetilde{u}^2,\pa_x(n\widetilde{u}))+\frac{1}{2}(\widetilde{u}^2,\pa_x(n u^r))
+(\widetilde{u}\pa_x\widetilde{u}+u^r\pa_x\widetilde{u}+\widetilde{u}\pa_xu^r,n\widetilde{u})
=(\widetilde{u}^2,n\pa_xu^r),
\end{split}
\end{equation}
which is a good term due to $n>0$ and $\pa_x u^r>0$.
We now turn to estimate the fourth term on the left hand side of \eqref{ztu.ip}.
In light of \eqref{trho} and \eqref{tphy} and by integration by parts, we have
\begin{equation}\label{ztu.ip2}
\begin{split}
-(\pa_x\widetilde{\phi}, n\widetilde{u})=\left(\widetilde{\phi}, -\pa_t\pa_x^2\widetilde{\phi}+\pa_t\left[n^r\left(1-e^{-\widetilde{\phi}}\right)\right]-\pa_t\pa_x^2\phi^r-\pa_x(\widetilde{n}u^r)\right).
\end{split}
\end{equation}
Next we compute the right hand side of \eqref{ztu.ip2} term by term. It is obvious to see that
\begin{equation}\label{phy.eg1.}
(\widetilde{\phi}, -\pa_t\pa_x^2\widetilde{\phi})=\frac{1}{2}\frac{d}{dt}(\pa_x\widetilde{\phi}, \pa_x\widetilde{\phi}).
\end{equation}
For the second term on the right hand side of \eqref{ztu.ip2}, we first get from the Taylor's formula with an integral remainder
that
\begin{equation}\label{taylor}
1-e^{-\widetilde{\phi}}=\widetilde{\phi}-\frac{1}{2}\widetilde{\phi}^2+
\underbrace{\int_0^{-\widetilde{\phi}}\frac{(\widetilde{\phi}+\varrho)^2}{2}e^{-\varrho} d\varrho}_{I_2}.
\end{equation}
Then it follows that
\begin{equation*}
\begin{split}
\left(\widetilde{\phi}, \pa_t\left[n^r\left(1-e^{-\widetilde{\phi}}\right)\right]\right)
=&\left(\widetilde{\phi}, \pa_t(n^r\widetilde{\phi})\right)
-\frac{1}{2}\left(\widetilde{\phi}, \pa_t(n^r\widetilde{\phi}^2)\right)
+\left(\widetilde{\phi}, \pa_t(n^rI_2)\right).
\end{split}
\end{equation*}
Since $[n^r, u^r]$ satisfies \eqref{Re.ME2.}, one has
\begin{equation}\label{phi.eg.zo.}
\left(\widetilde{\phi}, \pa_t(n^r\widetilde{\phi})\right)=\frac{1}{2}\frac{d}{dt}\left(\widetilde{\phi}^2, n^r\right)+\frac{1}{2}\left(\widetilde{\phi}^2, \pa_tn^r\right)
=\frac{1}{2}\frac{d}{dt}\left(\widetilde{\phi}^2, n^r\right)\underbrace{-\frac{1}{2}\left(\widetilde{\phi}^2, \pa_x(n^ru^r)\right)}_{I_3},
\end{equation}
\begin{equation*}
-\frac{1}{2}\left(\widetilde{\phi}, \pa_t(n^r\widetilde{\phi}^2)\right)=-\frac{1}{3}\frac{d}{dt}\left(\widetilde{\phi}^3, n^r\right)-\frac{1}{6}\left(\widetilde{\phi}^3, \pa_tn^r\right)
=-\frac{1}{3}\frac{d}{dt}\left(\widetilde{\phi}^3, n^r\right)+\underbrace{\frac{1}{6}\left(\widetilde{\phi}^3, \pa_x(n^ru^r)\right)}_{I_4},
\end{equation*}
and
$$
\left(\widetilde{\phi}, \pa_t(n^rI_2)\right)=
\underbrace{-\left(\widetilde{\phi}, \pa_x(n^ru^r)I_2\right)+\left(\widetilde{\phi}, n^r\pa_tI_2\right)}_{I_5}.
$$
We remark that $I_3$ and $I_4$ can not be directly controlled for the time being and they will be treated by cancelation with other terms later on.
Since $\|\widetilde{\phi}(t,x)\|_{L^{\infty}}\leq C$, it follows that
\begin{equation}\label{I3.com.}
I_2\thicksim \widetilde{\phi}^3,\ \ \pa_t I_2=\pa_t\widetilde{\phi}\int_0^{-\widetilde{\phi}}(\widetilde{\phi}+\varrho)e^{-\varrho} d\varrho
\thicksim\pa_t\widetilde{\phi}\,\widetilde{\phi}^2.
\end{equation}
In addition, we get from \eqref{trho} and \eqref{tphy} that
\begin{multline}\label{ptphyL2}
\|\pa_t\widetilde{\phi}\|^2+\|\pa_t\pa_x\widetilde{\phi}\|^2\leq
C\left\{\|\pa_t\widetilde{n}\|^2+\|\widetilde{\phi}\pa_tn^r\|^2+\|\pa_t\pa_x^2\phi^r\|^2\right\}\\
\leq C\left\{\|\pa_x(\widetilde{n} \widetilde{u})\|^2+\|\pa_x(\widetilde{n}u^r)\|^2+\|\pa_x(n^r\widetilde{u})\|^2\right\}+C\left\{\|\widetilde{\phi}\pa_tn^r\|^2+\|\pa_t\pa_x^2\phi^r\|^2\right\}.
\end{multline}
Substituting  \eqref{I3.com.} and \eqref{ptphyL2} into $I_5$, applying Lemma \ref{cl.Re.Re2.} and H\"older's inequality as well as the Cauchy-Schwarz  inequality with $0<\eta<1$, one has
\begin{equation}\label{I5}
\begin{split}
{|I_5|}\lesssim& C_\eta\left\|\widetilde{\phi}^3\right\|^2+\left\|\widetilde{\phi}\pa_x[n^r,u^r]\right\|^2+\eta\|\pa_t\widetilde{\phi}\|^2
\\
\lesssim& \eps_0C_\eta\|\pa_x\widetilde{\phi}\|^2+\eta\left\|\pa_x[\widetilde{n},\widetilde{u}]\right\|^2
+\|\pa_x[n^r,u^r]\|^2_{L^\infty}\|[\widetilde{n},\widetilde{u},\widetilde{\phi}]\|^2
+\left\|\pa^{2}_x\left(\frac{\pa_x(n^ru^r)}{n^r}\right)\right\|^2\\
\lesssim& \max\{\eps_0C_\eta,\eta\}\left\|\pa_x[\widetilde{n},\widetilde{u},\widetilde{\phi}]\right\|^2
+(1+t)^{-2}\|[\widetilde{n},\widetilde{u},\widetilde{\phi}]\|^2
+\eps(1+t)^{-2},
\end{split}
\end{equation}
where the following Sobolev's inequality
\begin{equation}\label{sob.ine.}
\|f\|_{L^\infty}\leq\sqrt{2}\|f\|^{1/2}\|\pa_xf\|^{1/2} \ \  \textrm{for any}\, \ f\in H^1,
\end{equation}
has been used to obtain the bound
$$
\left\|\widetilde{\phi}^3\right\|^2\leq C\|\pa_x\widetilde{\phi}\|^2\|\widetilde{\phi}\|^4.
$$
By integration by parts and employing Lemma \ref{cl.Re.Re2.} and the Cauchy-Schwarz  inequality with $0<\eta<1$, we see that
the third term on the right hand side of \eqref{ztu.ip2}
can be dominated by
\begin{equation}\label{phi.diss1.}
\left(\widetilde{\phi}, -\pa_t\pa_x^2\phi^r\right)\leq\eta\|\pa_x\widetilde{\phi}\|^2+C_\eta\|\pa_t\pa_x\phi^r\|^2
\leq\eta\|\pa_x\widetilde{\phi}\|^2+\eps C_\eta(1+t)^{-2}.
\end{equation}
As to the last term on the right hand side of \eqref{ztu.ip2}, we have from \eqref{tphy} that
\begin{equation}\label{lt3.7}
\begin{split}
\left(\widetilde{\phi}, -\pa_x(\widetilde{n}u^r)\right)
&=-\left(\widetilde{\phi}, \pa_x\widetilde{n}u^r\right)-\left(\widetilde{\phi}, \widetilde{n}\pa_xu^r\right)\\
&=\left(\widetilde{\phi}, \left(-\pa_x^3\widetilde{\phi}
+\pa_x\left[n^r\left(1-e^{-\widetilde{\phi}}\right)\right]-\pa_x^3\phi^r\right)u^r\right)\\
&\quad +\left(\widetilde{\phi}, \left(-\pa_x^2\widetilde{\phi}
+\left[n^r\left(1-e^{-\widetilde{\phi}}\right)\right]-\pa_x^2\phi^r\right)\pa_xu^r\right)\\
&=\underbrace{-\frac{1}{2}\left(\pa_xu^r, \left(\pa_x\widetilde{\phi}\right)^2\right)+\left(\pa_x\widetilde{\phi},\pa_x^2\phi^ru^r\right)}_{I_6}
+\underbrace{\left(\widetilde{\phi},
\pa_x\left[n^r\left(1-e^{-\widetilde{\phi}}\right)\right]u^r\right)}_{I_7}\\
&\quad +\underbrace{\left(\widetilde{\phi}, \left[n^r\left(1-e^{-\widetilde{\phi}}\right)\right]\pa_xu^r\right)}_{I_8},
\end{split}
\end{equation}
where the last identity holds true due to the following identities:
$$
\left(\widetilde{\phi}, -\pa_x^3\widetilde{\phi}
u^r\right)+\left(\widetilde{\phi}, -\pa_x^2\widetilde{\phi}
\pa_x u^r\right)=-\frac{1}{2}\left(\pa_xu^r, \left(\pa_x\widetilde{\phi}\right)^2\right)
$$
and
$$
\left(\widetilde{\phi}, -\pa_x^3\phi^ru^r\right)+\left(\widetilde{\phi}, -\pa_x^2\phi^r\pa_xu^r\right)=\left(\pa_x\widetilde{\phi},\pa_x^2\phi^ru^r\right).
$$
It is straightforward to see that $|I_6|$ is dominated by
$$
\eta\|\pa_x\widetilde{\phi}\|^2+C_\eta(1+t)^{-2}\|\pa_x\widetilde{\phi}\|^2+C_\eta\eps(1+t)^{-2},
$$
according to {the} Cauchy-Schwarz  inequality with $0<\eta<1$ and Lemma \ref{cl.Re.Re2.}.

We now use \eqref{taylor} to expand $I_7$ and $I_8$ respectively as 
\begin{equation*}
\begin{split}
I_7
=&\left(\widetilde{\phi}, \pa_x(n^r\widetilde{\phi})u^r\right)
-\frac{1}{2}\left(\widetilde{\phi}, \pa_x(n^r\widetilde{\phi}^2)u^r\right)
+\left(\widetilde{\phi}, \pa_x(n^rI_2)u^r\right)\\
=&\left(\widetilde{\phi}^2, \pa_xn^ru^r\right)-\frac{1}{2}\left(\widetilde{\phi}^2, \pa_x(n^ru^r)\right)
-\frac{1}{2}\left(\widetilde{\phi}^3, \pa_xn^ru^r\right)+\frac{1}{3}\left(\widetilde{\phi}^3, \pa_x(n^ru^r)\right)
+\left(\widetilde{\phi}, \pa_x(n^rI_2)u^r\right),
\end{split}
\end{equation*}
and
\begin{equation*}
\begin{split}
I_8=\left(\widetilde{\phi}, n^r\widetilde{\phi}\pa_xu^r\right)
-\frac{1}{2}\left(\widetilde{\phi}, n^r\widetilde{\phi}^2\pa_xu^r\right)
+\left(\widetilde{\phi}, n^rI_2\pa_xu^r\right).
\end{split}
\end{equation*}
Owing to these, we see
\begin{equation}\label{I4589}
I_3+I_4+I_7+I_8=\left(\widetilde{\phi}, \pa_x(n^rI_2)u^r\right)
+\left(\widetilde{\phi}, n^rI_2\pa_xu^r\right)
=-\left(\pa_x\widetilde{\phi}, n^rI_2u^r\right)
\leq C \eps_0^2\|\pa_x\widetilde{\phi}\|^2.
\end{equation}
Recalling \eqref{ztu.ip2}, we thereby complete the estimate on the term $-(\pa_x\widetilde{\phi}, n\widetilde{u})$ in the way that
\begin{multline*}
\left|[-(\pa_x\widetilde{\phi}, n\widetilde{u})]-\left[\frac{1}{2}\frac{d}{dt}(\pa_x\widetilde{\phi},\pa_x \widetilde{\phi})+\frac{1}{2}\frac{d}{dt}(\widetilde{\phi}^2,n^r)-\frac{1}{3}\frac{d}{dt}(\widetilde{\phi}^3,n^r)\right]\right|\\
\lesssim (\eta+\eps_0C_\eta +\eps_0^2) {\left\|\pa_x [\widetilde{n}, \widetilde{u},\widetilde{\phi}]\right\|^2}
+C_\eta (1+t)^{-2} {\left\| [\widetilde{n}, \widetilde{u},\widetilde{\phi}]\right\|^2}+\eps C_\eta (1+t)^{-2}.
\end{multline*}

It now remains to estimate the second term on the right hand side of \eqref{ztu.ip}.
Letting $0<\sigma<1$, by applying H\"older's inequality, Young's inequality, Lemma \ref{cl.Re.Re2.} and the Sobolev  inequality \eqref{sob.ine.}, one obtains
\begin{equation}\label{lrt3.5}
\begin{split}
{|(\pa_x^2u^r,\widetilde{u})|}\lesssim& \|\pa_x^2u^r\|_{L^{1}}\|\widetilde{u}\|_{L^{\infty}}
\lesssim \eps^{\frac{\sigma}{1+\sigma}}(1+t)^{-1}\|\widetilde{u}\|^{\frac{1-\sigma}{1+\sigma}}_{L^{\infty}}
\|\widetilde{u}\|^{\frac{2\sigma}{1+\sigma}}\\
\lesssim&\eps^{\frac{\sigma}{1+\sigma}}(1+t)^{-1}\|\pa_x\widetilde{u}\|^{\frac{1-\sigma}{2(1+\sigma)}}
\|\widetilde{u}\|^{\frac{1+3\sigma}{2(1+\sigma)}}\\
\lesssim&\eps^{\frac{\sigma}{1+\sigma}}\left\{\|\pa_x\widetilde{u}\|^2+(1+t)^{-1-\frac{1-\sigma}{3+5\sigma}}
\|\widetilde{u}\|^{\frac{2(1+3\sigma)}{3+5\sigma}}\right\}\\
\lesssim&\eps^{\frac{\sigma}{1+\sigma}}\left\{\|\pa_x\widetilde{u}\|^2+(1+t)^{-\frac{5+3\sigma}{4+4\sigma}}
+(1+t)^{-\frac{3+5\sigma}{2(1+3\sigma)}}\|\widetilde{u}\|^{2}\right\}.
\end{split}
\end{equation}
Finally by putting \eqref{ztu.ip}, \eqref{ztu.ip1}, \eqref{phy.eg1.}, \eqref{phi.eg.zo.}, \eqref{I5}, \eqref{phi.diss1.}, \eqref{lt3.7}, \eqref{I4589}
and \eqref{lrt3.5} together and letting $\eps>0, \eps_0>0$ and $0<\eta<1$ be suitably small, we have
\begin{equation}\label{ztu.sum1}
\begin{split}
\frac{d}{dt} &(\widetilde{u}, n\widetilde{u})+2\frac{d}{dt} (n, \psi)+\frac{d}{dt}\left(\widetilde{\phi}^2, n^r\right)+\frac{d}{dt}(\pa_x\widetilde{\phi}, \pa_x\widetilde{\phi})\\&-\frac{2}{3}\frac{d}{dt}\left(\widetilde{\phi}^3, n^r\right)
+2(\widetilde{u}^2,n\pa_xu^r)+\la\|\pa_x\widetilde{u}\|^2\\
\lesssim & (\eps_0 C_\eta+\eta)\left\|\pa_x[\widetilde{n},\widetilde{\phi}]\right\|^2
+ C_\eta(1+t)^{-2}\|[\widetilde{n},\widetilde{u},\widetilde{\phi},\pa_x\widetilde{\phi}]\|^2\\
&+\eps^{\frac{\sigma}{1+\sigma}}(1+t)^{-\frac{3+5\sigma}{2(1+3\sigma)}}\|\widetilde{u}\|^{2}
+\eps^{\frac{\sigma}{1+\sigma}}(1+t)^{-\frac{5+3\sigma}{4+4\sigma}}
+\eps C_\eta(1+t)^{-2}.
\end{split}
\end{equation}

\medskip
\noindent{\bf Step 2.} {\it The dissipation of $\pa_x\widetilde{n}$ and $\pa_x\widetilde{\phi}$.}

We first differentiate \eqref{tphy} and \eqref{trho} in $x$ to obtain
\begin{equation}\label{d.phi}
\pa_x^3\widetilde{\phi}=\pa_x\widetilde{n}+\pa_xn^r\left(1-e^{-\widetilde{\phi}}\right)
+n^re^{-\widetilde{\phi}}\pa_x\widetilde{\phi}-\pa_x^3\phi^r,
\end{equation}
and
\begin{equation}\label{d.rho}
\pa_t\pa_x\widetilde{n}+\pa_x u\pa_x\widetilde{n}+u\pa^2_x\widetilde{n}+\pa_xn\pa_x\widetilde{u}
+n\pa^2_x\widetilde{u}+\widetilde{n}\pa^2_xu^r
+\pa_x\widetilde{n}\pa_xu^r+\widetilde{u}\pa^2_xn^r
+\pa_xn^r\pa_x\widetilde{u}=0.
\end{equation}
Then taking the inner products of \eqref{d.phi}, \eqref{d.rho} and \eqref{tu} with $\pa_x\widetilde{\phi}$, $\frac{\pa_x\widetilde{n}}{n^2}$
and $\pa_x\widetilde{n}$ with  respect to $x$ over $\R$, respectively, one has
\begin{equation}\label{d.phi.mp}
-(\pa_x^3\widetilde{\phi},\pa_x\widetilde{\phi})+(\pa_x\widetilde{n},\pa_x\widetilde{\phi})
+\left(\pa_xn^r\left(1-e^{-\widetilde{\phi}}\right),\pa_x\widetilde{\phi}\right)
+(n^re^{-\widetilde{\phi}}\pa_x\widetilde{\phi},\pa_x\widetilde{\phi})-(\pa_x^3\phi^r,\pa_x\widetilde{\phi})=0,
\end{equation}
\begin{multline}\label{d.rho.mp}
\left(\pa_t\pa_x\widetilde{n}, \frac{\pa_x\widetilde{n}}{n^2}\right)+\left(\pa^2_x\widetilde{u},\frac{\pa_x\widetilde{n}}{n}\right)\\
+\left(\pa_x u\pa_x\widetilde{n}+u\pa^2_x\widetilde{n}+\pa_xn\pa_x\widetilde{u}
+\widetilde{n}\pa^2_xu^r
+\pa_x\widetilde{n}\pa_xu^r+\widetilde{u}\pa^2_xn^r
+\pa_xn^r\pa_x\widetilde{u}, \frac{\pa_x\widetilde{n}}{n^2}\right)=0,
\end{multline}
and
\begin{multline}\label{tu.mp}
(\pa_t \widetilde{u},\pa_x\widetilde{n})+(u\pa_x u-u^r\pa_x u^r, \pa_x\widetilde{n})+A\left(\frac{\pa_xn}{n}-\frac{\pa_xn^r}{n^r}, \pa_x\widetilde{n}\right)\\
-(\pa_x\widetilde{n}, \pa_x\widetilde{\phi})-\left(\pa_x^2\widetilde{u},\frac{\pa_x\widetilde{n}}{n}\right)
-\left(\pa_x^2u^r,\frac{\pa_x\widetilde{n}}{n}\right)=0.
\end{multline}
The summation of \eqref{d.phi.mp}, \eqref{d.rho.mp} and \eqref {tu.mp} implies
\begin{equation}\label{ztrho.sum1}
\begin{split}
\frac{d}{dt} &(\widetilde{u}, \pa_x\widetilde{n})+\frac{1}{2}\frac{d}{dt} ((\pa_x\widetilde{n})^2, n^{-2})+\left(\pa_xu^r, (\pa_x\widetilde{n})^2n^{-2}\right)+A\left(n^{-1}, (\pa_x\widetilde{n})^2\right)
\\&+\left(n^re^{-\widetilde{\phi}}\pa_x\widetilde{\phi},\pa_x\widetilde{\phi}\right)+(\pa^2_x\widetilde{\phi},\pa^2_x\widetilde{\phi})\\
= & -\left(\pa_xn^r\left(1-e^{-\widetilde{\phi}}\right),\pa_x\widetilde{\phi}\right)+(\pa_x^3\phi^r,\pa_x\widetilde{\phi})
-\left((\pa_x\widetilde{n})^2, n^{-3}\pa_tn\right)\\
&-\left(\pa_x u\pa_x\widetilde{n}+u\pa^2_x\widetilde{n}+\pa_xn\pa_x\widetilde{u}
+\widetilde{n}\pa^2_xu^r+\widetilde{u}\pa^2_xn^r
+\pa_xn^r\pa_x\widetilde{u}, \frac{\pa_x\widetilde{n}}{n^2}\right)\\
&+(\widetilde{u}, \pa_t \pa_x\widetilde{n})-(u\pa_x u-u^r\pa_x u^r, \pa_x\widetilde{n})
+A\left((nn^r)^{-1}\widetilde{n}\pa_xn^r, \pa_x\widetilde{n}\right)+\left(\pa_x^2u^r,\frac{\pa_x\widetilde{n}}{n}\right),
\end{split}
\end{equation}
which is further equal to
\begin{equation*}
\begin{split}
&\underbrace{ -\left(\pa_xn^r\left(1-e^{-\widetilde{\phi}}\right),\pa_x\widetilde{\phi}\right)}_{I_9}
+\underbrace{(\pa_x^3\phi^r,\pa_x\widetilde{\phi})}_{I_{10}}
+\underbrace{\frac{1}{2}\left(\pa_x u\pa_x\widetilde{n},\frac{\pa_x\widetilde{n}}{n^2}\right)}_{I_{11}}\\
&-\underbrace{\left(\pa_xn\pa_x\widetilde{u}
+\widetilde{n}\pa^2_xu^r+\widetilde{u}\pa^2_xn^r
+\pa_xn^r\pa_x\widetilde{u}, \frac{\pa_x\widetilde{n}}{n^2}\right)}_{I_{12}}\\
&+\underbrace{(\widetilde{u}, \pa_t \pa_x\widetilde{n})-(u\pa_x u-u^r\pa_x u^r, \pa_x\widetilde{n})}_{I_{13}}
+\underbrace{A\left((nn^r)^{-1}\widetilde{n}\pa_xn^r, \pa_x\widetilde{n}\right)}_{I_{14}}+\underbrace{\left(\pa_x^2u^r,\frac{\pa_x\widetilde{n}}{n}\right)}_{I_{15}},
\end{split}
\end{equation*}
where the following fact
$$
-\left((\pa_x\widetilde{n})^2, n^{-3}\pa_tn\right)
-\left(\pa_x u\pa_x\widetilde{n}+u\pa^2_x\widetilde{n}, \frac{\pa_x\widetilde{n}}{n^2}\right)=
\frac{1}{2}\left(\pa_x u\pa_x\widetilde{n},\frac{\pa_x\widetilde{n}}{n^2}\right)
$$
has been used to deduce $I_{11}$.

We now estimate terms $I_{j}$ $(9\leq j\leq15)$ as follows. For $I_9$, $I_{10}$, $I_{14}$ and $I_{15}$, it is direct to obtain
$$
{|I_{9}|}\leq \eta\|\pa_x\widetilde{\phi}\|^2+C_\eta\|\pa_xn^r\|^2_{L^\infty}\|\widetilde{\phi}\|^2
\leq \eta\|\pa_x\widetilde{\phi}\|^2+C_\eta(1+t)^{-2}\|\widetilde{\phi}\|^2,
$$
$$
{|I_{10}|}\leq \eta\|\pa_x\widetilde{\phi}\|^2+C_\eta\|\pa_x^3\phi^r\|^2
\leq \eta\|\pa_x\widetilde{\phi}\|^2+C_\eta\eps^3 (1+t)^{-2},
$$
$$
{|I_{14}|}\leq \eta\|\pa_x\widetilde{n}\|^2+C_\eta\|\pa_xn^r\|^2_{L^\infty}\|\widetilde{n}\|^2
\leq \eta\|\pa_x\widetilde{n}\|^2+C_\eta(1+t)^{-2}\|\widetilde{n}\|^2,
$$
$$
{|I_{15}|}\leq \eta\|\pa_x\widetilde{n}\|^2+C_\eta\eps(1+t)^{-2}.
$$
{For $I_{11}$ one has
\begin{equation*}
\begin{split}
I_{11}=& \frac{1}{2} \left(\pa_x u^r\pa_x\widetilde{n},\frac{\pa_x \widetilde{n}}{n^2}\right)
+ \frac{1}{2} \left(\pa_x \widetilde{u}\pa_x\widetilde{n},\frac{\pa_x \widetilde{n}}{n^2}\right),
\end{split}
\end{equation*}
and
\begin{equation*}
\begin{split}
 \left|\left(\pa_x \widetilde{u}\pa_x\widetilde{n},\frac{\pa_x \widetilde{n}}{n^2}\right)\right|
\leq&\eta \|\pa_x \widetilde{n}\|^2 +C_\eta \|\pa_x \widetilde{u} \pa_x \widetilde{n}\|^2
\leq\eta \|\pa_x \widetilde{n}\|^2 +C_\eta \|\pa_x \widetilde{n}\|^2 (\|\pa_x \widetilde{u}\|^2+\|\pa_x^2 \widetilde{u}\|^2)
\\ \leq&(\eta +C_\eta \eps_0^2)\|\pa_x\widetilde{n}\|^2 +C_\eta \eps_0^2\|\pa_x^2 \widetilde{u}\|^2.
\end{split}
\end{equation*}
As to $I_{12}$, it follows that}
\begin{equation*}
\begin{split}
{|I_{12}|}\leq& \eta\|\pa_x\widetilde{n}\|^2+C_\eta\|\pa_x\widetilde{n}\|^2(\|\pa_x\widetilde{u}\|^2+\|\pa^2_x\widetilde{u}\|^2)
+C_\eta\|\pa_xn^r\|_{L^\infty}^2 \|\pa_x\widetilde{u}\|^2\\
&+C_\eta (\|\widetilde{n}\|_{L^\infty}^2+\|\widetilde{u}\|_{L^\infty}^2) (\|\pa_x^2 u^r\|^2+\|\pa_x^2n^r\|^2)\\
\leq&(\eta+C_\eta \eps_0^2)\|[\pa_x\widetilde{n},\pa_x\widetilde{u}]\|^2+C_\eta \eps_0^2 \|\pa_x^2\widetilde{u} \|^2+C_\eta\eps_0^2\eps(1+t)^{-2}.
\end{split}
\end{equation*}
To estimate $I_{13}$, we first notice that
\begin{eqnarray*}
I_{13} &=&(\pa_x \widetilde{u},-\pa_t \widetilde{n})-\big((\widetilde{u}+u^r)\pa_x (\widetilde{u}+u^r) -u^r\pa_x u^r,\pa_x \widetilde{n}\big)\\
&=&\big(\pa_x \widetilde{u}, \pa_x (n u)-\pa_x (n^r u^r)\big)- (\widetilde{u} \pa_x \widetilde{u} +\widetilde{u} \pa_x u^r
+{u^r \pa_x\widetilde{u}},\pa_x \widetilde{n})\\
&=&((\pa_x \widetilde{u})^2, \widetilde{n})+(\pa_x \widetilde{u},\widetilde{n}\pa_x u^r) +(\pa_x \widetilde{u},\pa_x n^r \widetilde{u}) +(\pa_x \widetilde{u}, n^r \pa_x \widetilde{u})-(\widetilde{u}\pa_x u^r,\pa_x \widetilde{n}),
\end{eqnarray*}
where the right-hand terms can be estimated as
\begin{multline*}
{|((\pa_x \widetilde{u})^2, \widetilde{n})|+|(\pa_x \widetilde{u},\pa_x n^r \widetilde{u})| +|(\pa_x \widetilde{u}, n^r \pa_x \widetilde{u})|}\\
\leq C \eps_0 \|\pa_x \widetilde{u}\|^2 +\left[\eta \|\pa_x \widetilde{u}\|^2 + C_\eta (1+t)^{-2}\|\widetilde{u}\|^2\right] +n_+\|\pa_x \widetilde{u}\|^2,
\end{multline*}
and
\begin{eqnarray*}
|(\pa_x \widetilde{u},\widetilde{n}\pa_x u^r)|+|(\widetilde{u}\pa_x u^r,\pa_x \widetilde{n})|
&\leq& \eta \|\pa_x[\widetilde{n},\widetilde{u}]\|^2 + C_\eta  \|[\widetilde{n},\widetilde{u}]\pa_x u^r\|^2\\
&\leq&  \eta \|\pa_x[\widetilde{n},\widetilde{u}]\|^2 +C_\eta (1+t)^{-2} \|[\widetilde{n},\widetilde{u}]\|^2.
\end{eqnarray*}
Therefore it follows that
\begin{equation*}
{|I_{13}|\lesssim  (\eta+\eps_0 )\|\pa_x[\widetilde{n},\widetilde{u}]\|^2+n_+\|\pa_x \widetilde{u}\|^2
+C_\eta (1+t)^{-2} \|[\widetilde{n},\widetilde{u}]\|^2.}
\end{equation*}
%
%

By plugging the above estimates back to \eqref{ztrho.sum1}  and letting $0<\eta<1$ and $\eps_0>0$ be suitably small, we obtain
\begin{eqnarray}
&&\frac{d}{dt} (\widetilde{u}, \pa_x\widetilde{n})+\frac{1}{2}\frac{d}{dt} ((\pa_x\widetilde{n})^2, n^{-2})+\frac{1}{2}\left(\pa_xu^r, (\pa_x\widetilde{n})^2n^{-2}\right)\notag\\
&&\quad +\la\left[\left(n^{-1}, (\pa_x\widetilde{n})^2\right)+\left(n^re^{-\widetilde{\phi}}\pa_x\widetilde{\phi},\pa_x\widetilde{\phi}\right)\right]+(\pa^2_x\widetilde{\phi},\pa^2_x\widetilde{\phi})\notag\\
&&\leq \max\{1,n_+\}\|\pa_x\widetilde{u}\|^2 +C\eps_0^2\|\pa_x^2\widetilde{u}\|^2
+ C(1+t)^{-2} {\left\|[\widetilde{n},\widetilde{u},\widetilde{\phi}]\right\|^2}+C\eps (1+t)^{-2}.\label{s2.f}
\end{eqnarray}
Further taking the summation of $2\times$\eqref{s2.f} and $C_1\times$\eqref{ztu.sum1} for a suitably large constant $C_1>1$ and also letting $\eps_0>0$ and $0<\eta<1$ in \eqref{ztrho.sum1} be suitably small gives
\begin{equation}\label{ztru.sum1}
\begin{split}
\frac{d}{dt} &2(\widetilde{u}, \pa_x\widetilde{n})+\frac{d}{dt} ((\pa_x\widetilde{n})^2, n^{-2})
+C_1\left\{\frac{d}{dt} (\widetilde{u}, n\widetilde{u})+2\frac{d}{dt} (n, \psi)\right\}\\
&+C_1\left\{\frac{d}{dt}\left(\widetilde{\phi}^2, n^r\right)+\frac{d}{dt}(\pa_x\widetilde{\phi}, \pa_x\widetilde{\phi})-\frac{2}{3}\frac{d}{dt}\left(\widetilde{\phi}^3, n^r\right)\right\}\\
&
+2C_1(\pa_xu^r,n\widetilde{u}^2)+\left(\pa_xu^r, (\pa_x\widetilde{n})^2n^{-2}\right)
+(\pa^2_x\widetilde{\phi},\pa^2_x\widetilde{\phi})
\\&+\la\left\{(\pa_x\widetilde{u},\pa_x\widetilde{u})+\left(n^{-1}, (\pa_x\widetilde{n})^2\right)+\left(n^re^{-\widetilde{\phi}}\pa_x\widetilde{\phi},\pa_x\widetilde{\phi}\right)\right\}
\\
\lesssim & \eps_0\|\pa_x^2\widetilde{u}\|^{2}+(1+t)^{-2}\left\|[\widetilde{n},\widetilde{u},\widetilde{\phi},\pa_x\widetilde{\phi}]\right\|^2
+(1+t)^{-\frac{3+5\sigma}{2(1+3\sigma)}}\|\widetilde{u}\|^{2}+\eps^{\frac{\sigma}{1+\sigma}}(1+t)^{-\frac{5+3\sigma}{4+4\sigma}}
+\eps(1+t)^{-2},
\end{split}
\end{equation}
where the large constant $C_1>1$ is chosen so as to also guarantee
$$
2(\widetilde{u}, \pa_x\widetilde{n})+((\pa_x\widetilde{n})^2, n^{-2})
+C_1(\widetilde{u}, n\widetilde{u})\thicksim \|\pa_x\widetilde{n}\|^2+\|\widetilde{u}\|^2.
$$
From Gronwall's inequality, it follows further from \eqref{ztru.sum1} that
\begin{equation}\label{ztru.sum2}
\begin{split}
\sup\limits_{0\leq t\leq T} &\Big\{2(\widetilde{u}, \pa_x\widetilde{n})+ ((\pa_x\widetilde{n})^2, n^{-2})
+C_1(\widetilde{u}, n\widetilde{u})+2C_1(n, \psi)
+C_1\left(\widetilde{\phi}^2, n^r\right)\\&\qquad\qquad+C_1(\pa_x\widetilde{\phi}, \pa_x\widetilde{\phi})-\frac{2}{3}C_1\left(\widetilde{\phi}^3, n^r\right)\Big\}\\
&+\int_0^T\left\{(\pa_xu^r,n\widetilde{u}^2)+\left(\pa_xu^r, (\pa_x\widetilde{n})^2n^{-2}\right)
+(\pa^2_x\widetilde{\phi},\pa^2_x\widetilde{\phi})\right\}dt
\\&+\int_0^T\left\{(\pa_x\widetilde{u},\pa_x\widetilde{u})+\left(n^{-1}, (\pa_x\widetilde{n})^2\right)+\left(n^re^{-\widetilde{\phi}}\pa_x\widetilde{\phi},\pa_x\widetilde{\phi}\right)\right\}dt
\\
\lesssim & \|[\widetilde{n},\widetilde{u},\widetilde{\phi}](0)\|^2+\|\pa_x [\widetilde{n},\widetilde{\phi}](0)\|^2
+\eps_0\int_0^T\|\pa^2_x\widetilde{u}\|^2dt+\eps^{\frac{\sigma}{1+\sigma}}.
\end{split}
\end{equation}

\medskip
\noindent{\bf Step 3.} {\it Dissipation of $\pa_x^2\widetilde{u}$.}

By differentiating \eqref{tu} in $x$ and taking the inner products of the resulting identity with $\pa_x\widetilde{u}$
with  respect to $x$ over $\R$, we have
\begin{multline}\label{tu.uip}
(\pa_t \pa_x\widetilde{u},\pa_x\widetilde{u})+(\pa_x(\widetilde{u}\pa_x \widetilde{u}+\widetilde{u}\pa_x u^r+\pa_x \widetilde{u}u^r), \pa_x\widetilde{u})+A\left(\pa_x\left(\frac{\pa_xn}{n}-\frac{\pa_xn^r}{n^r}\right), \pa_x\widetilde{u}\right)\\
+(-\pa^2_x\widetilde{\phi}, \pa_x\widetilde{u})+\left(-\pa_x\left(\frac{\pa_x^2\widetilde{u}}{n}\right),\pa_x\widetilde{u}\right)
+\left(-\pa_x\left(\frac{\pa_x^2u^r}{n}\right),\pa_x\widetilde{u}\right)=0.
\end{multline}
We estimate each inner product term on the left as follows. First, the first term is equal to $\frac{1}{2}\frac{d}{dt}(\pa_x \widetilde{u},\pa_x \widetilde{u})$.
The second term is computed as
\begin{equation*}
(\pa_x(\widetilde{u}\pa_x \widetilde{u}+\widetilde{u}\pa_x u^r+\pa_x \widetilde{u}u^r), \pa_x\widetilde{u})={ \frac{3}{2} (\pa_x u^r, (\pa_x \widetilde{u})^2)
-(\widetilde{u}\pa_x^2 \widetilde{u},\pa_x \widetilde{u}) + (\widetilde{u}\pa_x^2 u^r, \pa_x \widetilde{u})},
\end{equation*}
where the first  term on the right are good and the rest two terms are bounded as
\begin{multline*}
{|(\widetilde{u}\pa_x^2 \widetilde{u},\pa_x \widetilde{u})| + |(\widetilde{u}\pa_x^2 u^r, \pa_x \widetilde{u})|}
\leq \eta  (\|\pa_x \widetilde{u} \|^2+\|\pa_x^2  \widetilde{u}\|^2)  +C_\eta \| \widetilde{u}\|_{L^\infty}^2 (\|\pa_x \widetilde{u} \|^2+\|\pa_x^2 u^r\|^2)\\
\leq \eta  (\|\pa_x \widetilde{u} \|^2+\|\pa_x^2  \widetilde{u}\|^2)  +C_\eta \eps_0^2\|\pa_x \widetilde{u} \|^2+C_\eta \eps_0^2 \eps (1+t)^{-2},
\end{multline*}
with an arbitrary constant $0<\eta<1$. The third term can be rewritten as
\begin{multline*}
A\left(\pa_x\left(\frac{\pa_xn}{n}-\frac{\pa_xn^r}{n^r}\right), \pa_x\widetilde{u}\right)
=A\left(\frac{\pa_x\widetilde{n}}{n}-\frac{\widetilde{n} \pa_x n^r}{n n^r},-\pa_x^2 \widetilde{u}\right)\\
=A\left(\frac{\pa_x\widetilde{n}}{n},-\pa_x^2 \widetilde{u}\right)
+A\left(-\frac{\pa_x \widetilde{n} \pa_x n^r +\widetilde{n} \pa_x^2 n^r}{n n^r},\pa_x\widetilde{u}\right)
+A\left(\frac{\widetilde{n} \pa_x n^r (\pa_x n n^r +n \pa_x n^r)}{ (nn^r)^2},\pa_x\widetilde{u}\right),
\end{multline*}
where the first inner product on the right is bounded by $\eta \|\pa_x^2 \widetilde{u}\|^2+C_\eta \|\pa_x \widetilde{n}\|^2$, the second one is bounded by
\begin{equation*}
\eta \|\pa_x \widetilde{u}\|^2 +C_\eta (\|\pa_xn^r\|_{L^\infty}^2\|\pa_x \widetilde{n}\|^2 +\|\widetilde{n}\|_{L^\infty}^2 \|\pa_x^2 n^r\|^2)
\leq \eta \|\pa_x \widetilde{u}\|^2 +C\eps \|\pa_x\widetilde{n}\|^2 +C\eps_0^2 \eps (1+t)^{-2},
\end{equation*}
and the final one is bounded by
\begin{eqnarray*}
&&\eta \|\pa_x \widetilde{u}\|^2 +C_\eta \|\widetilde{n}\|_{L^\infty}^2\|\pa_x n^r\|_{L^\infty}^2 \|\pa_x \widetilde{n}\|^2 \|n^r\|_{L^\infty}^2
+C_\eta \|\widetilde{n}\|_{L^\infty}^2 \|\pa_x n^r\|_{L^\infty}^2 \|\pa_x n^r\|^2\\
&&\leq \eta \|\pa_x \widetilde{u}\|^2 + C_\eta\eps_0^2\|\pa_x \widetilde{n}\|^2 +C_\eta \|\widetilde{n}\|\cdot \|\pa_x \widetilde{n}\|\cdot \|\pa_x n^r\|\cdot \|\pa_x^2n^r\|\cdot \|\pa_xn^r\|^2\\
&&\leq  \eta \|\pa_x \widetilde{u}\|^2 + C_\eta\eps_0^2\|\pa_x \widetilde{n}\|^2 +C_\eta\|\pa_x \widetilde{n} \|^2 +C_\eta  \|\widetilde{n}\|^2 \|\pa_x n^r\|^6\|\pa_x^2n^r\|^2\\
&&\leq   \eta \|\pa_x \widetilde{u}\|^2 + C_\eta\eps_0^2\|\pa_x \widetilde{n}\|^2 +C_\eta\|\pa_x \widetilde{n} \|^2 +C_\eta\eps_0^2 \eps (1+t)^{-5}.
\end{eqnarray*}
For the fourth term on the left hand side of \eqref{tu.uip}, it is direct to see
\begin{equation*}
|(\pa^2_x\widetilde{\phi}, \pa_x\widetilde{u})|\leq \eta {\|\pa_x\widetilde{u}\|^2} +C_\eta \|\pa_x^2\widetilde{\phi}\|^2
\end{equation*}
The fifth term is a good term given by $(n^{-1},(\pa_x^2\widetilde{u} )^2)$. The final term is bounded as
\begin{equation*}
\left|\left(-\pa_x\left(\frac{\pa_x^2u^r}{n}\right),\pa_x\widetilde{u}\right)\right|\leq \eta \|\pa_x^2\widetilde{u}\|^2 +C_\eta\left\|\frac{\pa_x^2u^r}{n}\right\|^2\leq  \eta \|\pa_x^2\widetilde{u}\|^2 +C_\eta\eps (1+t)^{-2}.
\end{equation*}
Plugging these estimates back to \eqref{tu.uip} gives
\begin{multline}\label{tu.uip.p1}
{\frac{1}{2}\frac{d}{dt}(\pa_x\widetilde{u},\pa_x\widetilde{u})+\frac{3}{2}(\pa_x u^r, (\pa_x\widetilde{u})^2)
+\left(\frac{\pa_x^2\widetilde{u}}{n},\pa^2_x\widetilde{u}\right)}\\
\leq
\eta\left\{\|\pa_x \widetilde{u}\|^2+\|\pa^2_x \widetilde{u}\|^2\right\}+C_\eta\eps_0^2\|\pa_x \widetilde{u}\|^2
+C_\eta \eps(1+t)^{-2}+C_\eta(\|\pa_x \widetilde{n}\|^2+\|\pa^2_x\widetilde{\phi}\|^2).
\end{multline}
Integrating \eqref{tu.uip.p1} with respect to $t$ over $[0,T]$ and letting $0<\eta<1$ be suitably small, one further has
\begin{multline}\label{tu.uip.p2}
\sup\limits_{0\leq t\leq T}(\pa_x\widetilde{u},\pa_x\widetilde{u})+\int_0^T(\pa_x u^r, (\pa_x\widetilde{u})^2)dt
+\int_0^T\left(\frac{\pa_x^2\widetilde{u}}{n},\pa^2_x\widetilde{u}\right)dt\\
\lesssim \|\pa_x \widetilde{u}(0)\|^2+(\eta+C_\eta\eps_0^2)\int_0^T\|\pa_x \widetilde{u}\|^2dt
+C_\eta\int_0^T(\|\pa_x \widetilde{n}\|^2+\|\pa^2_x\widetilde{\phi}\|^2)dt
+C_\eta\eps.
\end{multline}
Thus combining \eqref{ztru.sum2} and \eqref{tu.uip.p2} yields
\begin{equation}
\begin{split}
&\sup\limits_{0\leq t\leq T}\Big\{2(\widetilde{u}, \pa_x\widetilde{n})+((\pa_x\widetilde{n})^2, n^{-2})
+C_1(\widetilde{u}, n\widetilde{u})+2C_1(n, \psi)\\
&\qquad\qquad+
C_1(\widetilde{\phi}^2, n^r)+C_1(\pa_x\widetilde{\phi}, \pa_x\widetilde{\phi})
-\frac{2}{3}C_1(\widetilde{\phi}^3, n^r)\Big\}+\sup\limits_{0\leq t\leq T}(\pa_x\widetilde{u},\pa_x\widetilde{u})\\
&\quad +\int_0^T\left\{(\pa_xu^r,n\widetilde{u}^2)+\left(\pa_xu^r, (\pa_x\widetilde{n})^2n^{-2}\right)
+(\pa_x u^r, (\pa_x\widetilde{u})^2)\right\}dt
\\&\quad+\int_0^T\left\{(\pa_x\widetilde{u},\pa_x\widetilde{u})+\left(n^{-1}, (\pa_x\widetilde{n})^2\right)+(\pa^2_x\widetilde{\phi},\pa^2_x\widetilde{\phi})\right\}dt\\
&\quad+\int_0^T
\left\{\left(n^re^{-\widetilde{\phi}}\pa_x\widetilde{\phi},\pa_x\widetilde{\phi}\right)
+\left(\frac{\pa_x^2\widetilde{u}}{n},\pa^2_x\widetilde{u}\right)\right\}dt
\\
&\leq C\left\|[\widetilde{n}, \widetilde{u}, \widetilde{\phi}](0,x)\right\|^2_{H^1}+C\eps^{\frac{\sigma}{1+\sigma}}.
\end{split}
\label{total.eng2.}
\end{equation}
Notice that
\begin{multline*}
2(\widetilde{u}, \pa_x\widetilde{n})+((\pa_x\widetilde{n})^2, n^{-2})
+C_1(\widetilde{u}, n\widetilde{u})+2C_1(n, \psi)+C_1
\left(\widetilde{\phi}^2, n^r\right)\\+C_1(\pa_x\widetilde{\phi}, \pa_x\widetilde{\phi})-\frac{2}{3}C_1\left(\widetilde{\phi}^3, n^r\right)
\thicksim \left\|\left[\widetilde{n}, \widetilde{\phi}\right]\right\|^2_{H^1}+\|\widetilde{u}\|^2,
\end{multline*}
and
\begin{multline*}
(\pa_x\widetilde{u},\pa_x\widetilde{u})+\left(n^{-1}, (\pa_x\widetilde{n})^2\right)
+(\pa^2_x\widetilde{\phi},\pa^2_x\widetilde{\phi})
+\left(n^re^{-\widetilde{\phi}}\pa_x\widetilde{\phi},\pa_x\widetilde{\phi}\right)
+\left(\frac{\pa_x^2\widetilde{u}}{n},\pa^2_x\widetilde{u}\right)\\
\thicksim \left\|[\pa_x\widetilde{n}, \pa_x\widetilde{u}, \pa_x\widetilde{\phi}, \pa^2_x\widetilde{\phi}, \pa^2_x\widetilde{u}]\right\|^2,
\end{multline*}
according to the a priori assumption \eqref{pri.ass.} and the fact that $n_+>n_->0$, and furthermore the Poisson equation \eqref{tphy} implies that for any $t\geq 0$,
\begin{eqnarray*}
\|\widetilde{\phi}(t)\|_{H^1}^2&\leq &C\|\widetilde{n}(t)\|^2 +C\|\pa_x^2 \phi^r(t)\|^2\\
&\leq &C\|\widetilde{n}(t)\|^2 +C \|\pa_x^2 n^r(t)\|^2 +C\|\pa_xn^r(t)\|_{L^4}^4\\
&\leq &C\|\widetilde{n}(t)\|^2 +C \|\pa_x^2 n^r(t)\|^2 +C\|\pa_x n^r(t)\|^3\|\pa_x^2n^r(t)\|\\
&\leq &C\|\widetilde{n}(t)\|^2 +C\eps (1+t)^{-2}+C\eps^2 (1+t)^{-1},
\end{eqnarray*}
and hence in particular,
\begin{equation*}
\|\widetilde{\phi}(0)\|_{H^1}^2\leq C \left[\|\widetilde{n}(0)\|^2+ \eps\right].
\end{equation*}
The above observations together with \eqref{total.eng2.} give \eqref{total.eng1.}. This then completes the proof of Proposition \ref{a.pri.}.
\end{proof}

We are now in a position to complete the

\begin{proof}[Proof of Theorem \ref{main.res.}.]
The existence of the solution follows from the standard continuity argument based on the local existence and the a priori
estimate in Proposition \ref{a.pri.}. Therefore, it suffices to show the large time behavior of the solution as $t\rightarrow+\infty$.
For this, we begin with 
the following estimates
\begin{equation}\label{latm1}
\lim\limits_{t\rightarrow+\infty}\left\|\pa_x[\widetilde{n}, \widetilde{u}](t)\right\|^2\rightarrow 0,
\end{equation}
and
\begin{equation}\label{latm2}
\lim\limits_{t\rightarrow+\infty}\left\{\left\|\sqrt{n^re^{-\widetilde{\phi}}}\pa_x\widetilde{\phi}(t)\right\|^2
+\left\|\pa^2_x\widetilde{\phi}(t)\right\|^2\right\}\rightarrow 0.
\end{equation}
Indeed, {from \eqref{d.rho}, \eqref{tu.uip}, \eqref{total.eng.} and Lemma \ref{cl.Re.Re2.}, one can show that
\begin{equation*}
\begin{split}
\int_{0}^{+\infty}\left|\frac{d}{dt}\left\|\pa_x\widetilde{n}\right\|^2\right|dt
=2\int_{0}^{+\infty}\left|\left(\pa_t\pa_x\widetilde{n},\pa_x\widetilde{n}\right)\right|dt<+\infty,
\end{split}
\end{equation*}}
and
$$
\int_{0}^{+\infty}\left|\frac{d}{dt}\left\|\pa_x\widetilde{u}\right\|^2\right|dt
=\frac{1}{2}\int_{0}^{+\infty}\left|\left(\pa_t\pa_x\widetilde{u},\pa_x\widetilde{u}\right)\right|dt<+\infty.
$$
Then \eqref{latm1} follows from the above two inequalities and \eqref{total.eng.}.

The estimates for \eqref{latm2} will be a little different from the usual one as in \cite{MN86}. By \eqref{d.phi} and
applying Lemma \ref{cl.Re.Re2.} and \eqref{total.eng.}, we see that
\begin{equation}\label{latm3}
\int_{0}^{+\infty}\left|\frac{d}{dt}\left[\left(n^re^{-\widetilde{\phi}}\pa_x\widetilde{\phi},\pa_x\widetilde{\phi}\right)
+\left(\pa^2_x\widetilde{\phi},\pa^2_x\widetilde{\phi}\right)\right]\right|dt
\lesssim C+\int_{0}^{+\infty}\|\pa_t\widetilde{\phi}\|^2dt.
\end{equation}
On the other hand, from \eqref{ptphyL2}, one has
\begin{equation}\label{latm4}
\int_{0}^{+\infty}\left\|\pa_t\pa_x\widetilde{\phi}\right\|^2dt+\int_{0}^{+\infty}
\left\|\pa_t\widetilde{\phi}\right\|^2dt
<+\infty.
\end{equation}
Thus \eqref{latm3}, \eqref{latm4} and \eqref{total.eng.} give \eqref{latm2}.

We consequently get from \eqref{latm1}, \eqref{latm2} and the Sobolev inequality \eqref{sob.ine.} that
\begin{equation*}
\lim\limits_{t\rightarrow+\infty}\sup\limits_{x\in\R}\left|\left[n(t,x)-n^{r}(x,t), u(t,x)- u^{r}(x,t), \phi(t,x)+\ln n^{r}(t,x)\right]\right|=0.
\end{equation*}
Furthermore, by the construction of the smooth approximation function of the rarefaction wave, in terms of $(iii)$ in Lemma \ref{cl.Re.Re2.}, we obtain the desired asymptotic behavior of the solution
\begin{equation*}
\lim\limits_{t\rightarrow+\infty}\sup\limits_{x\in\R}\left|\left[n(t,x)-n^{R}\left(\frac{x}{t}\right), u(t,x)- u^{R}\left(\frac{x}{t}\right), \phi(t,x)+\ln n^{R} \left(\frac{x}{t}\right)\right]\right|=0.
\end{equation*}
Hence \eqref{asp} holds true. This ends the proof of Theorem \ref{main.res.}.
\end{proof}

\section{Two-fluid case}
%

In this section, we will apply the similar technique developed in the previous section to study the  stability of rarefaction waves for the Cauchy problem on the two-fluid NSP system  \eqref{NSP2}, \eqref{I.D.2}, \eqref{I.D.2-c}, \eqref{bd2}. The new difficulty appears due to the fact that the electric field $\pa_x \phi$ is no longer $L^2$ integrable in space and time due to the structure of the Poisson equation in \eqref{NSP2}. It will be seen from the later proof that the new trouble term
$$
\frac{1}{2}\int_\R (\pa_x\widetilde{\phi})^2 \pa_xu^r\,dx
$$
can be controlled by taking the difference of two momentum equations with different weights so as to balance the different masses $m_\al$ $(\al=i,e)$ of two  fluids. Therefore, the structure of two-fluid model indeed plays a key role in the stability analysis of nontrivial rarefaction waves.

%

%
%
%
%
\subsection{Approximate rarefaction waves}

The large time solution of \eqref{NSP2} is assumed to be the rarefaction wave $[n^r,u^r]$ under the quasineutral assumption. Thus, we may set
$$
n_i=n_e=n^r,\ \ u_i=u_e=u^r.
$$
By neglecting the second-order derivatives $\pa_x^2u^r$ and $\pa_x^2\phi^r$ which decay in time faster than other low-order terms, $[n^r,u^r]$ is expected to satisfy the following equations:
\begin{eqnarray}
&&\pa_t n^r+\pa_x(n^ru^r)=0,\label{2eulern}\\
&&m_in^r(\pa_t u^r+u^r\pa_x u^r)+T_i\pa_x n^r-n^r\pa_x\phi^r=0,\label{2euleri}\\
&&m_en^r(\pa_t u^r+u^r\pa_x u^r)+T_e\pa_x n^r+n^r\pa_x\phi^r=0.\label{2eulere}
\end{eqnarray}
It is straightforward to verify that system \eqref{2eulern}, \eqref{2euleri}, \eqref{2eulere} holds true if  $[n^r,u^r]$ satisfies
\begin{equation}\label{2rare}
\left\{
\begin{array}{l}
\dis \pa_t n^r+\pa_x(n^ru^r)=0,\\[3mm]
\dis n^r(\pa_t u^r+u^r\pa_x u^r)+\frac{T_i+T_e}{m_i+m_e}\pa_x n^r=0,\\[3mm]
\dis \phi^r=\frac{T_im_e-T_em_i}{m_i+m_e}\ln n^r.
\end{array}\right.
\end{equation}
%
%
Therefore, the large-time asymptotic equations of \eqref{NSP2} of
the two-fluid NSP system are expected to take the form of the following quasineutral Euler equations
\begin{equation}\label{2Euler}
\left\{\begin{array}{l}
\dis \pa_t\overline{n}+\pa_x(\overline{n}~\overline{u})=0,\\[3mm]
\dis \overline{n}(\pa_t \overline{u}+\overline{u}\pa_x \overline{u})+\frac{T_i+T_e}{m_i+m_e}\pa_x \overline{n}=0,
\end{array}
\right.
\end{equation}
with the potential function $\overline{\phi}$ in large time determined by
\begin{equation*}
\overline{\phi}=\frac{T_im_e-T_em_i}{m_i+m_e} \ln \overline{n}.
\end{equation*}
Initial data for \eqref{2Euler} are given by
\begin{eqnarray}\label{2ME.I.D.}
\begin{array}{rll}
\begin{split}
[\overline{n}, \overline{u}](0,x)&=[\overline{n}_0,\overline{u}_0](x)=\left\{\begin{array}{rll}[n_-,u_-],&\ \ x>0,\\[3mm]
[n_+,u_+],&\ \ x<0,
\end{array}
\right.\\[2mm]
\overline{\phi}(0,x)&=\overline{\phi}_0=\frac{T_im_e-T_em_i}{m_i+m_e}\ln \overline{n}_0.
\end{split}
\end{array}
\end{eqnarray}
Similarly as before,  \eqref{2Euler}
have two characteristics
$$
\la_1=u-c,\ \ \la_2=u+c,\ \ \mbox{with}\ c=\sqrt{\frac{T_i+T_e}{m_i+m_e}},
$$
where we have used the same notation for $c$ as in \eqref{1cha}. In the completely same way as in the previous section, concerning the Riemann problem \eqref{2Euler} and \eqref{2ME.I.D.} on the quasineutral Euler system, one can construct the $2$-rarefaction wave $[n^{R},u^{R}](x/t)$ from the Burgers' equation and further construct its smooth approximation $[n^r,u^r]$ satisfying the properties given in Lemma \ref{cl.Re.Re2.}.

\subsection{Stability of the rarefaction wave}
In this subsection, we study the stability of the approximate rarefaction wave $[n^{r}, u^{r}]$
for the Cauchy problem  \eqref{NSP2}, \eqref{I.D.2}, \eqref{I.D.2-c}, \eqref{bd2}. For this purpose, let us define the perturbation
\begin{equation*}
[\widetilde{n}_\al, \widetilde{u}_\al, \widetilde{\phi}]=[\widetilde{n}_{i,e}, \widetilde{u}_{i,e}, \widetilde{\phi}]=[n_\al-n^r, u_\al-u^r, \phi-\phi^r](t,x),\quad \al=i,e,
\end{equation*}
where $\phi^r$ is defined by $n^r$ in terms of the third equation of \eqref{2rare}.
Then $[\widetilde{n}_\al,\widetilde{u}_\al, \widetilde{\phi}]$ satisfies
\begin{eqnarray}
&&\pa_t\widetilde{n}_i+\pa_x(n_i u_i)-\pa_x(n^ru^r)=0,\label{2tni}\\
&&m_i\left(\pa_t \widetilde{u}_i+u_i\pa_x u_i-u^r\pa_x u^r\right)+T_i(\pa_x\ln n_i-\pa_x\ln n^r)
-\pa_x\widetilde{\phi}=\frac{\pa_x^2u_i}{n_i},\label{2tui}\\
&&\pa_t\widetilde{n}_e+\pa_x(n_e u_e)-\pa_x(n^ru^r)=0,\label{2tne}\\
&&m_e\left(\pa_t \widetilde{u}_e+u_e\pa_x u_e-u^r\pa_x u^r\right)+T_e(\pa_x\ln n_e-\pa_x\ln n^r)
+\pa_x\widetilde{\phi}=\frac{\pa_x^2u_e}{n_e},\label{2tue}\\
&&\pa_x^2\widetilde{\phi}=\widetilde{n}_i-\widetilde{n}_e-\pa_x^2\phi^r.\label{2tphy}
\end{eqnarray}
Here we have set $\mu_i=\mu_e=1$ without loss of generality. Initial data for $[\widetilde{n}_\al,\widetilde{u}_\al]$ are given by
\begin{equation}
\label{2tid}
[\widetilde{n}_\al,\widetilde{u}_\al](0,x)
=[{n}_{\al0}(x)-n^r(0,x),{u}_{\al 0}-u^r(0,x)],\quad \al=i,e.
\end{equation}
From now on it is always assumed that for any given $t\geq 0$, $\widetilde{\phi}(t,x)$ is determined in terms of the elliptic equation \eqref{2tphy} under the boundary condition that $\widetilde{\phi}(t,x)$ tends to zero as $x$ goes to $\pm\infty$.

\begin{theorem}\label{2main.res.}
Let ${[n_+, u_+]}\in R_2(n_-, u_-)$, where $R_2(n_-,u_-)$ is given in \eqref{def.r2} with $n_->0$ and $c=\sqrt{\frac{T_i+T_e}{m_i+m_e}}$. Assume
\begin{equation*}
\phi_{\pm}=\frac{T_i m_e-T_em_i}{m_i+m_e}\ln n_{\pm}.
\end{equation*}
There are constants $\eps_0>0$, $0<\si<1$, and $C_0>0$ such that if
$$
\sum\limits_{\al=i,e}\left\|[\widetilde{n}_{\al0},\widetilde{u}_{\al0}]\right\|^2_{H^1}+\|\pa_x\widetilde{\phi}_0\|^2+\eps\leq \eps_0^2,
$$
where $\eps>0$ is the parameter appearing in \eqref{cl.Re.cons.}, then the Cauchy problem of
the NSP system \eqref{2tni}, \eqref{2tui}, \eqref{2tne}, \eqref{2tue}, \eqref{2tphy} and \eqref{2tid}
admits a unique global solution $[\widetilde{n}_{i,e},\widetilde{u}_{i,e},\widetilde{\phi}](t,x)$ satisfying
\begin{multline}\label{2total.eng.}
\sup\limits_{t\geq0} \left\{\sum\limits_{\al=i,e}\left\|\left[\widetilde{n}_{\al}, \widetilde{u}_{\al}\right](t)\right\|^2_{H^1}
+\left\|\pa_x\widetilde{\phi}(t)\right\|^2_{H^1}\right\}
+\sum\limits_{\al=i,e}\int_0^{+\infty}\left\|\sqrt{\pa_xu^r}\left[ \widetilde{u}_{\al}, \pa_x\widetilde{n}_{\al},\pa_x\widetilde{u}_{\al}\right](t)\right\|^2dt
\\+\sum\limits_{\al=i,e}\int_0^{+\infty}\left\|\left[\pa_x\widetilde{n}_{\al}, \pa_x\widetilde{u}_{\al}, \pa^2_x\widetilde{u}_{\al}\right](t)\right\|^2dt
+\int_0^{+\infty}\left\| \pa^2_x\widetilde{\phi}(t)\right\|^2dt
\leq C_0\eps_0^{\frac{2\sigma}{1+\sigma}}.
\end{multline}
Moreover, the solution to the original Cauchy problem  \eqref{NSP2}, \eqref{I.D.2}, \eqref{I.D.2-c}, \eqref{bd2} tends time-asymptotically to the rarefaction in the sense that
\begin{equation}\label{2asp}
\lim\limits_{t\rightarrow+\infty}\sup\limits_{x\in\R}\left|\left[n_\al(t,x)-n^{R}\left(\frac{x}{t}\right), u_\al(t,x)- u^{R}\left(\frac{x}{t}\right)\right]\right|=0,\ \ \al=i,\ e,
\end{equation}
and
\begin{equation}\label{2asp-ad}
\lim\limits_{t\rightarrow+\infty}\sup\limits_{x\in\R}\left|\pa_x\left[\phi(t,x)-\frac{T_im_e-T_em_i}{m_i+m_e}\ln n^{r} (t,x)\right]\right|=0.
\end{equation}
\end{theorem}

We remark that smallness on $\|\pa_x \widetilde{\phi}_0\|$ is required in Theorem \ref{2main.res.}. In addition, due to the weaker dissipation of the potential function in the two-fluid case, it is not clear to verify the uniform convergence of $\phi(t,x)$ to
$$
\phi^R(x/t)=\frac{T_i m_e-T_em_i}{m_i+m_e}\ln n^R(x/t),
$$
as time goes to infinity. However, \eqref{2asp-ad} still implies the large-time convergence of the electric field $\pa_x \phi$ to the smooth approximate profile $\pa_x \phi^r$.

As in the previous section,
the local existence of the solution $[\widetilde{n},\widetilde{u},\widetilde{\phi}]$  of the
NSP system \eqref{2tni}, \eqref{2tui}, \eqref{2tne}, \eqref{2tue}, \eqref{2tphy} and \eqref{2tid} can be proved by the standard iteration method and its proof is omitted for brevity.
As to the proof of Theorem \ref{2main.res.}, it suffices to show the following a priori estimate.

\begin{proposition}\label{2a.pri.}
Assume that all the conditions in Theorem \ref{2main.res.} hold. There are constants $\eps_0>0$, $0<\si<1$, and $C_0>0$ such that if the smooth solution $[\widetilde{n},\widetilde{u},\widetilde{\phi}]$ to the Cauchy problem
\eqref{2tni}, \eqref{2tui}, \eqref{2tne}, \eqref{2tue}, \eqref{2tphy} and \eqref{2tid} on $0\leq t\leq T$ for $T>0$ satisfies
\begin{equation}\label{2pri.ass.}
\sup\limits_{0\leq t\leq T}\left\{{\sum\limits_{\al=i,e}\left\|\left[\widetilde{n}_{\al}, \widetilde{u}_{\al}\right](t)\right\|^2_{H^1}}+\|\pa_x\widetilde{\phi}(t)\|^2\right\}+\eps\leq \eps_0^2,
\end{equation}
then it holds that
\begin{multline}\label{2total.eng1.}
\sup\limits_{0\leq t\leq T} \left\{\sum\limits_{\al=i,e}\left\|\left[\widetilde{n}_{\al}, \widetilde{u}_{\al}\right](t)\right\|^2_{H^1}
+\left\|\pa_x\widetilde{\phi}(t)\right\|^2_{H^1}\right\}
+\sum\limits_{\al=i,e}\int_0^{T}\left\|\sqrt{\pa_xu^r}\left[ \widetilde{u}_{\al}, \pa_x\widetilde{n}_{\al},\pa_x\widetilde{u}_{\al}\right](t)\right\|^2dt
\\+\sum\limits_{\al=i,e}\int_0^{T}\left\|\left[\pa_x\widetilde{n}_{\al}, \pa_x\widetilde{u}_{\al}, \pa^2_x\widetilde{u}_{\al}\right](t)\right\|^2dt
+\int_0^{T}\left\| \pa^2_x\widetilde{\phi}(t)\right\|^2dt
\\ \leq C_0\left(\sum_{\al=i,e}\left\|[\widetilde{n}_{\al 0}, \widetilde{u}_{\al 0}]\right\|^2_{H^1}+\|\pa_x\widetilde{\phi}_0\|^2+\eps^{\frac{\sigma}{1+\sigma}}\right).
\end{multline}
\end{proposition}

\subsection{A priori estimate}

\begin{proof}[Proof of Proposition \ref{2a.pri.}]
As for obtaining Proposition \ref{a.pri.}, we divide it by three steps as follows.

\medskip
\noindent{\bf Step 1.} {\it Zero-order energy estimates.}

As in \eqref{def.psi}, set
\begin{equation*}
\psi_\al=\psi(n_{\al},n^r)=T_\al\int_{n^r}^{n_{\al}}\frac{s-n^r}{s^2}ds,\quad \al=i,e.
\end{equation*}
Note that $\psi(n_{\al},n^r)$ is equivalent with  $|n_\al-n^r|^2$ for $|n_\al-n^r|\leq C$.
Multiplying  the first equation of \eqref{NSP2} and \eqref{2tui} by $\psi_i=\psi(n_i,n^r)$ and $n_i\widetilde{u}_i$ respectively,
 integrating the resulting equations with  respect to $x$ over $\R$ and taking the summation, one obtains
\begin{equation}\label{2zeui}
\frac{m_i}{2}\frac{d}{dt} (\widetilde{u}_i, n_i\widetilde{u}_i)+\frac{d}{dt} (n_i, \psi_i)
+m_i(\widetilde{u}_i^2,n_i\pa_xu^r)
-(\pa_x\widetilde{\phi}, n_i\widetilde{u}_i)=(\pa_x^2\widetilde{u}_i,\widetilde{u}_i)+(\pa_x^2u^r,\widetilde{u}_i),
\end{equation}
where we have used the following identities
$$
-(n_i,\pa_t\psi_i)+(\pa_x(n_i u_i),\psi_i)+T_i\left(\pa_x\left(\ln n_i-\ln n^r\right),n_i\widetilde{u}_i\right)=0,
$$
and
\begin{multline*}
-\frac{1}{2}(\widetilde{u}_i^2, \pa_tn_i)+(u_i\pa_x u_i-u^r\pa_x u^r,n_i\widetilde{u}_i)\\
=\frac{1}{2}(\widetilde{u}_i^2,\pa_x(n_i\widetilde{u}_i))+\frac{1}{2}(\widetilde{u}_i^2,\pa_x(n_i u^r))
+(\widetilde{u}_i\pa_x\widetilde{u}_i+u^r\pa_x\widetilde{u}_i+\widetilde{u}_i\pa_xu^r,n_i\widetilde{u}_i)
=(\widetilde{u}_i^2,n_i\pa_xu^r).
\end{multline*}
In a similar way, from the third equation of $\eqref{NSP2}$ and \eqref{2tue}, we have
\begin{equation}\label{2zeue}
\frac{m_e}{2}\frac{d}{dt} (\widetilde{u}_e, n_e\widetilde{u}_e)+\frac{d}{dt} (n_e, \psi_e)
+m_e(\widetilde{u}_e^2,n_e\pa_xu^r)
+(\pa_x\widetilde{\phi}, n_e\widetilde{u}_e)=(\pa_x^2\widetilde{u}_e,\widetilde{u}_e)+(\pa_x^2u^r,\widetilde{u}_e).
\end{equation}
We note that since $\pa_x u^r>0$ and $n_\al>0$,  $(\widetilde{u}_\al^2,n_\al\pa_xu^r)$ are good terms  for $\al=i$ and $e$, which plays a crucial role in the zero-order energy estimates.
Furthermore, we get from \eqref{2tni}, \eqref{2tne} and \eqref{2tphy} that
\begin{eqnarray*}
&&-(\pa_x\widetilde{\phi}, n_i\widetilde{u}_i)+(\pa_x\widetilde{\phi}, n_e\widetilde{u}_e)\\
&&=
(\widetilde{\phi}, \pa_x(n_i\widetilde{u}_i))-(\widetilde{\phi}, \pa_x(n_e\widetilde{u}_e))\\
&&=-(\widetilde{\phi}, \pa_t(\widetilde{n}_i-\widetilde{n}_e))-(\widetilde{\phi}, \pa_x((\widetilde{n}_i-\widetilde{n}_e)u^r))
\\
&&=-(\widetilde{\phi}, \pa_t\pa_x^2\widetilde{\phi})-(\widetilde{\phi}, \pa_t\pa_x^2\phi^r)
-(\widetilde{\phi}, \pa_x(\pa_x^2\widetilde{\phi}u^r))-(\widetilde{\phi}, \pa_x(\pa_x^2\phi^ru^r))\\
&&=\frac{1}{2}\frac{d}{dt}(\pa_x\widetilde{\phi},\pa_x\widetilde{\phi})+(\pa_x\widetilde{\phi}, \pa_t\pa_x\phi^r)
-\frac{1}{2}((\pa_x\widetilde{\phi})^2, \pa_xu^r)+(\pa_x\widetilde{\phi}, \pa_x^2\phi^ru^r).
\end{eqnarray*}
Therefore, with the above identity, the summation of \eqref{2zeui} and \eqref{2zeue} implies
\begin{equation}\label{2zeui+ue}
\begin{split}
\frac{d}{dt} &\left\{\frac{m_i}{2}(\widetilde{u}_i, n_i\widetilde{u}_i)+\frac{m_e}{2}(\widetilde{u}_e, n_e\widetilde{u}_e)\right\}
+\frac{d}{dt}\left\{(n_i, \psi_i)+(n_e, \psi_e)\right\}+\frac{1}{2}\frac{d}{dt}(\pa_x\widetilde{\phi},\pa_x\widetilde{\phi})\\
&\underbrace{+m_i(\widetilde{u}_i^2,n_i\pa_xu^r)+m_e(\widetilde{u}_e^2,n_e\pa_xu^r)}_{I_0}
+\|\pa_x[\widetilde{u}_i,\widetilde{u}_e]\|^2
\\=&\underbrace{\frac{1}{2}((\pa_x\widetilde{\phi})^2, \pa_xu^r)}_{I_1}
+\underbrace{(\pa_x^2u^r,\widetilde{u}_i)}_{I_2}+\underbrace{(\pa_x^2u^r,\widetilde{u}_e)}_{I_3}
\underbrace{-(\pa_x\widetilde{\phi}, \pa_t\pa_x\phi^r)}_{I_4}
\underbrace{-(\pa_x\widetilde{\phi}, \pa_x^2\phi^ru^r)}_{I_5}.
\end{split}
\end{equation}

We have to develop a little more delicate estimates on $I_1$, which are different from the case of the one-fluid ions considered in the previous section. The basic idea is to make use of the good term $I_0$ to absorb $I_1$,
since the latter one
can not be directly controlled.
For this, we choose two positive constants $\be$ and $\ga$ such that $m_i\be=m_e \ga$. Notice that physically the ratio $\ga/\be=m_i/m_e$ is large. Taking the inner product of \eqref{2tui} and \eqref{2tue} with
$$
\frac{\be}{2(\be+\ga)}\pa_x\widetilde{\phi}\pa_xu^r\ \text{and}\ -\frac{\ga}{2(\be+\ga)}\pa_x\widetilde{\phi}\pa_xu^r,
$$
respectively, and taking the summation of the resulting equations, one has
\begin{equation}\label{2de.el}
\begin{split}
-\frac{1}{2}&((\pa_x\widetilde{\phi})^2, \pa_xu^r)
\\=&-\frac{\be m_i}{2(\be+\ga)}\frac{d}{dt}(\widetilde{u}_i,\pa_x\widetilde{\phi}\pa_xu^r)
+\frac{\ga m_e}{2(\be+\ga)}\frac{d}{dt}(\widetilde{u}_e,\pa_x\widetilde{\phi}\pa_xu^r)
\\&+\underbrace{\frac{1}{2(\be+\ga)}(\be m_i\widetilde{u}_i-\ga m_e\widetilde{u}_e,\pa_t\pa_x\widetilde{\phi}\pa_xu^r)}_{I_6}
+\underbrace{\frac{1}{2(\be+\ga)}(\be m_i\widetilde{u}_i-\ga m_e\widetilde{u}_e,\pa_x\widetilde{\phi}\pa_t\pa_xu^r)}_{I_7}
\\&\underbrace{-\frac{1}{2(\be+\ga)}\left(\be m_i[u_i\pa_x u_i-u^r\pa_x u^r]-\ga m_e[u_e\pa_x u_e-u^r\pa_x u^r],\pa_x\widetilde{\phi}\pa_xu^r\right)}_{I_8}
\\&\underbrace{-\frac{1}{2(\be+\ga)}\left(\be T_i\pa_x\ln \frac{n_i}{n^r}-\gamma T_e\pa_x\ln \frac{n_e}{n^r},\pa_x\widetilde{\phi}\pa_xu^r\right)}_{I_9}+\underbrace{\frac{1}{2(\be+\ga)}\left(\be\frac{\pa^2_xu_i}{n_i}-\ga \frac{\pa^2_xu_e}{n_e},\pa_x\widetilde{\phi}\pa_xu^r\right)}_{I_{10}}.
\end{split}
\end{equation}
On the other hand, it follows from \eqref{2tphy} that
\begin{equation*}
\pa_t\pa_x\widetilde{\phi}=n_eu_e-n_iu_i-\pa_t\pa_x\phi^r=n_e\widetilde{u}_e-n_i\widetilde{u}_i
+(\widetilde{n}_e-\widetilde{n}_i)u^r{-\pa_t\pa_x\phi^r},
\end{equation*}
which implies
\begin{equation*}
\begin{split}
I_6=&-\frac{\be m_i}{2(\be+\ga)}((\widetilde{u}_i)^2,n_i\pa_xu^r)-\frac{\ga m_e}{2(\be+\ga)}((\widetilde{u}_e)^2,n_e\pa_xu^r)\\
&+\frac{\be m_i}{2(\be+\ga)}(\widetilde{u}_i\widetilde{u}_e,n_e\pa_xu^r)
+\frac{\ga m_e}{2(\be+\ga)}(\widetilde{u}_i\widetilde{u}_e,n_i\pa_xu^r)
\\&{-\frac{1}{2(\be+\ga)}(\be m_i\widetilde{u}_i-\ga m_e\widetilde{u}_e,\pa_t\pa_x\phi^r\pa_xu^r)}.
\end{split}
\end{equation*}
Consequently, it holds that
\begin{equation}\label{2I0+I6}
\begin{split}
\frac{I_0}{2}+I_6
=&\frac{1}{2(\be+\ga)}\int_{\R}\bigg\{\ga m_in_i(\widetilde{u}_i)^2
+(\be m_in_e+\ga m_en_i)\widetilde{u}_i\widetilde{u}_e
+\be m_en_e(\widetilde{u}_e)^2\bigg\}\pa_xu^r dx
\\=&\underbrace{\frac{1}{2(\be+\ga)}\int_{\R}\bigg\{\ga m_i\widetilde{n}_i(\widetilde{u}_i)^2
+(\be m_i\widetilde{n}_e+\ga m_e\widetilde{n}_i)\widetilde{u}_i\widetilde{u}_e
+\be m_e\widetilde{n}_e(\widetilde{u}_e)^2\bigg\}\pa_xu^r dx}_{I_{11}}
\\&+\underbrace{\frac{n^r}{2(\be+\ga)}\int_{\R}\bigg\{\ga m_i(\widetilde{u}_i)^2
+(\be m_i+\ga m_e)\widetilde{u}_i\widetilde{u}_e
+\be m_e(\widetilde{u}_e)^2\bigg\}\pa_xu^r dx}_{I_{12}}
\\&{\underbrace{-\frac{1}{2(\be+\ga)}\int_{\R}(\be m_i\widetilde{u}_i-\ga m_e\widetilde{u}_e)\pa_t\pa_x\phi^r\pa_xu^rdx}_{I_{13}}}.
\end{split}
\end{equation}
One can claim that $I_{12}\geq0$ due to the choice of constants $\be$ and $\ga$. Indeed, by using $\be m_i=\ga m_e$, the integrand of $I_{12}$ can be rewritten as
$$
(\sqrt{\ga m_i}\widetilde{u}_i+\sqrt{\be m_e}\widetilde{u}_e)^2\pa_xu^r,
$$
which is always non-negative for $\be>0$ and $\ga>0$.
With \eqref{2I0+I6} in hand, substituting \eqref{2de.el} into \eqref{2zeui+ue}, we arrive at
\begin{equation}\label{2zeui+ue2}
\begin{split}
\frac{d}{dt} &\left\{\frac{m_i}{2}(\widetilde{u}_i, n_i\widetilde{u}_i)+\frac{m_e}{2}(\widetilde{u}_e, n_e\widetilde{u}_e)\right\}
+\frac{d}{dt}\left\{(n_i, \psi_i)+(n_e, \psi_e)\right\}+\frac{1}{2}\frac{d}{dt}(\pa_x\widetilde{\phi},\pa_x\widetilde{\phi})
\\&-\frac{\be m_i}{2(\be+\ga)}\frac{d}{dt}(\widetilde{u}_i,\pa_x\widetilde{\phi}\pa_xu^r)
+\frac{\ga m_e}{2(\be+\ga)}\frac{d}{dt}(\widetilde{u}_e,\pa_x\widetilde{\phi}\pa_xu^r)\\
&+\frac{1}{2}\left\{m_i(\widetilde{u}_i^2,n_i\pa_xu^r)+m_e(\widetilde{u}_e^2,n_e\pa_xu^r)\right\}
+\|\pa_x[\widetilde{u}_i,\widetilde{u}_e]\|^2+I_{12}
\\=&I_2+I_3+I_4+I_5-I_7-I_8-I_9-I_{10}-I_{11}{-I_{13}}.
\end{split}
\end{equation}
We now turn to estimate the right hand side of \eqref{2zeui+ue2} term by term. Letting $0<\sigma<1$, by applying H\"older's inequality, Young's inequality and Lemma \ref{cl.Re.Re2.},  as for obtaining \eqref{lrt3.5}, one obtains
\begin{equation*}
\begin{split}
|I_2|+|I_3|
\lesssim&\eps^{\frac{\sigma}{1+\sigma}}\left\{\|\pa_x[\widetilde{u}_i,\widetilde{u}_e]\|^2+(1+t)^{-\frac{5+3\sigma}{4+4\sigma}}
+(1+t)^{-\frac{3+5\sigma}{2(1+3\sigma)}}\|[\widetilde{u}_i,\widetilde{u}_e]\|^{2}\right\}.
\end{split}
\end{equation*}
In the completely same way, it follows that
\begin{equation*}
\begin{split}
|I_4|+|I_5|
\lesssim\eps^{\frac{\sigma}{1+\sigma}}\left\{\|\pa^2_x\widetilde{\phi}\|^2+(1+t)^{-\frac{5+3\sigma}{4+4\sigma}}
+(1+t)^{-\frac{3+5\sigma}{2(1+3\sigma)}}\|\pa_x\widetilde{\phi}\|^{2}\right\}.
\end{split}
\end{equation*}
By integration by parts and employing Lemma \ref{cl.Re.Re2.} and Cauchy-Schwarz's inequality with $0<\eta<1$, we have
\begin{equation*}
\begin{split}
|I_7|
\lesssim\eta\|\pa_x[\widetilde{u}_i,\widetilde{u}_e,\pa_x\widetilde{\phi}]\|^2
+C_\eta(1+t)^{-2}\|[\widetilde{u}_i,\widetilde{u}_e,\pa_x\widetilde{\phi}]\|^{2}.
\end{split}
\end{equation*}
From Lemma \ref{cl.Re.Re2.} and Sobolev's inequality \eqref{sob.ine.} as well as Cauchy-Schwarz's inequality with $0<\eta<1$, it holds that
\begin{equation*}
\begin{split}
|I_8|+|I_9|
\lesssim\eta\|\pa_x[\widetilde{n}_i,\widetilde{n}_e,\widetilde{u}_i,\widetilde{u}_e]\|^2
+\eps_0\|\pa_x[\widetilde{u}_i,\widetilde{u}_e]\|^2
+C_\eta(1+t)^{-2}\|[\widetilde{n}_i,\widetilde{n}_e,\widetilde{u}_i,\widetilde{u}_e,\pa_x\widetilde{\phi}]\|^{2},
\end{split}
\end{equation*}
and
\begin{equation*}
\begin{split}
|I_{10}|
\lesssim(\eta+\eps)\|\pa_x[\widetilde{u}_i,\widetilde{u}_e]\|^2+\eps\|\pa^2_x\widetilde{\phi}\|^{2}
+C_\eta(1+t)^{-2}\|\pa_x\widetilde{\phi}\|^{2}+\eps C_\eta(1+t)^{-2}.
\end{split}
\end{equation*}
For the term $I_{11}$, it follows from the
Sobolev's inequality \eqref{sob.ine.} that
\begin{equation*}
\begin{split}
|I_{11}|
\lesssim\eps_0\left\{m_i(\widetilde{u}_i^2,n_i\pa_xu^r)+m_e(\widetilde{u}_e^2,n_e\pa_xu^r)\right\}.
\end{split}
\end{equation*}
As to the last term $I_{13}$, we get from Lemma \ref{cl.Re.Re2.} and H\"older's inequality that
$$
|I_{13}|
\lesssim (1+t)^{-2}\|[\widetilde{u}_i,\widetilde{u}_e]\|^{2}+\eps(1+t)^{-2}.
$$
Therefore, by plugging the above estimates 
into \eqref{2zeui+ue2} and letting $\eps>0, \eps_0>0$ and $0<\eta<1$ be suitably small, we have
\begin{equation}\label{2zeui+ue3}
\begin{split}
\frac{d}{dt} &\left\{\frac{m_i}{2}(\widetilde{u}_i, n_i\widetilde{u}_i)+\frac{m_e}{2}(\widetilde{u}_e, n_e\widetilde{u}_e)\right\}
+\frac{d}{dt}\left\{(n_i, \psi_i)+(n_e, \psi_e)\right\}+\frac{1}{2}\frac{d}{dt}(\pa_x\widetilde{\phi},\pa_x\widetilde{\phi})
\\&-\frac{\be m_i}{2(\be+\ga)}\frac{d}{dt}(\widetilde{u}_i,\pa_x\widetilde{\phi}\pa_xu^r)
+\frac{\ga m_e}{2(\be+\ga)}\frac{d}{dt}(\widetilde{u}_e,\pa_x\widetilde{\phi}\pa_xu^r)\\
&+\la\left\{m_i(\widetilde{u}_i^2,n_i\pa_xu^r)+m_e(\widetilde{u}_e^2,n_e\pa_xu^r)\right\}
+\la\|\pa_x[\widetilde{u}_i,\widetilde{u}_e]\|^2
\\ \lesssim &\eta\|\pa_x[\widetilde{n}_i,\widetilde{n}_e]\|^2+\eps^{\frac{\sigma}{1+\sigma}}\|\pa^2_x\widetilde{\phi}\|^2
+C_\eta (1+t)^{-2}\|[\widetilde{n}_i,\widetilde{n}_e,\widetilde{u}_i,\widetilde{u}_e,\pa_x\widetilde{\phi}]\|^{2}
\\&+\eps^{\frac{\sigma}{1+\sigma}}(1+t)^{-\frac{3+5\sigma}{2(1+3\sigma)}}\|[\widetilde{u}_i,\widetilde{u}_e]\|^{2}
+\eps^{\frac{\sigma}{1+\sigma}}(1+t)^{-\frac{5+3\sigma}{4+4\sigma}}
+\eps C_\eta (1+t)^{-2}.
\end{split}
\end{equation}

\medskip
\noindent{\bf Step 2.} {\it Dissipation of $\pa_x[\widetilde{n}_i,\widetilde{n}_e]$ and $\pa^2_x\widetilde{\phi}$.}

We first differentiate \eqref{2tphy}, \eqref{2tni} and \eqref{2tne} with respect to $x$ to obtain
\begin{equation}\label{2d.phi}
\pa_x^3\widetilde{\phi}=\pa_x(\widetilde{n}_i-\widetilde{n}_e)-\pa_x^3\phi^r,
\end{equation}
\begin{equation}\label{2d.ni}
\pa_t\pa_x\widetilde{n}_i+\pa_x u_i\pa_x\widetilde{n}_i+u_i\pa^2_x\widetilde{n}_i+\pa_x n_i\pa_x\widetilde{u}_i
+n_i\pa^2_x\widetilde{u}_i+\widetilde{n}_i\pa^2_xu^r
+\pa_x\widetilde{n}_i\pa_xu^r+\widetilde{u}_i\pa^2_xn^r
+\pa_xn^r\pa_x\widetilde{u}_i=0,
\end{equation}
and
\begin{equation}\label{2d.ne}
\pa_t\pa_x\widetilde{n}_e+\pa_x u_e\pa_x\widetilde{n}_e+u_e\pa^2_x\widetilde{n}_e+\pa_x n_e\pa_x\widetilde{u}_e
+n_e\pa^2_x\widetilde{u}_e+\widetilde{n}_e\pa^2_xu^r
+\pa_x\widetilde{n}_e\pa_xu^r+\widetilde{u}_e\pa^2_xn^r
+\pa_xn^r\pa_x\widetilde{u}_e=0.
\end{equation}
By taking the inner products of \eqref{2d.phi}, \eqref{2d.ni}, \eqref{2d.ne}, \eqref{2tui} and \eqref{2tue} with $-\pa_x\widetilde{\phi}$, $\frac{\pa_x\widetilde{n}_i}{n_i^2}$
, $\frac{\pa_x\widetilde{n}_e}{n_e^2}$, $\pa_x\widetilde{n}_i$ and $\pa_x\widetilde{n}_e$ with  respect to $x$ over $\R$, respectively, taking the summation of resulting five identities, and noticing the cancelation
$$
(\pa_x(\widetilde{n}_i-\widetilde{n}_e),-\pa_x\widetilde{\phi}) +(\pa_x\widetilde{\phi}, \pa_x\widetilde{n}_i)+(-\pa_x\widetilde{\phi}, \pa_x\widetilde{n}_e)=0,
$$
one has
\begin{equation}\label{2ztrho.sum1}
\begin{split}
\frac{d}{dt} &\sum\limits_{\al=i,e}m_\al(\widetilde{u}_\al, \pa_x\widetilde{n}_\al)+\frac{1}{2}\frac{d}{dt}\sum\limits_{\al=i,e} ((\pa_x\widetilde{n}_\al)^2, n_\al^{-2})+\sum\limits_{\al=i,e}\left(\pa_xu^r, (\pa_x\widetilde{n}_\al)^2n_\al^{-2}\right)\\
&+\sum\limits_{\al=i,e}T_\al\left(n_\al^{-1}, (\pa_x\widetilde{n}_\al)^2\right)
+(\pa^2_x\widetilde{\phi},\pa^2_x\widetilde{\phi})\\
= &(\pa_x^3\phi^r,\pa_x\widetilde{\phi})
-\sum\limits_{\al=i,e}\left((\pa_x\widetilde{n}_\al)^2, n_\al^{-3}\pa_tn_\al\right)\\
&-\sum\limits_{\al=i,e}(\pa_x u_\al\pa_x\widetilde{n}_\al+u_\al\pa^2_x\widetilde{n}_\al+\pa_xn_\al\pa_x\widetilde{u}_\al +\widetilde{n}_\al\pa^2_xu^r+\widetilde{u}_\al\pa^2_xn^r
+\pa_xn^r\pa_x\widetilde{u}_\al, \frac{\pa_x\widetilde{n}_\al}{n_\al^2})\\
&+\sum\limits_{\al=i,e}m_\al\left\{(\widetilde{u}_\al, \pa_t \pa_x\widetilde{n}_{\al})-(u_{\al}\pa_x u_{\al}-u^r\pa_x u^r, \pa_x\widetilde{n}_{\al})\right\}
\\&+\sum\limits_{\al=i,e}T_{\al}\left((n_{\al}n^r)^{-1}\widetilde{n}_{\al}\pa_xn^r, \pa_x\widetilde{n}_{\al}\right)+\sum\limits_{\al=i,e}\left(\pa_x^2u^r,\frac{\pa_x\widetilde{n}_{\al}}{n_{\al}}\right),
\end{split}
\end{equation}
which is further equal to
\begin{equation*}
\begin{split}
&\underbrace{(\pa_x^3\phi^r,\pa_x\widetilde{\phi})}_{J_1}
+\underbrace{\frac{1}{2}\sum\limits_{\al=i,e}\left(\pa_x u_\al\pa_x\widetilde{n}_\al,\frac{\pa_x\widetilde{n}_\al}{n_\al^2}\right)}_{J_{2}}\\
&-\underbrace{\sum\limits_{\al=i,e}\left(\pa_xn_\al\pa_x\widetilde{u}_\al
+\widetilde{n}_\al\pa^2_xu^r+\widetilde{u}_\al\pa^2_xn^r
+\pa_xn^r\pa_x\widetilde{u}_\al, \frac{\pa_x\widetilde{n}_\al}{n_\al^2}\right)}_{J_{3}}\\
&\underbrace{+\sum\limits_{\al=i,e}m_\al\left\{(\widetilde{u}_\al, \pa_t \pa_x\widetilde{n}_{\al})-(u_{\al}\pa_x u_{\al}-u^r\pa_x u^r, \pa_x\widetilde{n}_{\al})\right\}}_{J_4}
\\&\underbrace{+\sum\limits_{\al=i,e}T_{\al}\left((n_{\al}n^r)^{-1}\widetilde{n}_{\al}\pa_xn^r, \pa_x\widetilde{n}_{\al}\right)}_{J_5}+\underbrace{\sum\limits_{\al=i,e}\left(\pa_x^2u^r,\frac{\pa_x\widetilde{n}_{\al}}{n_{\al}}\right)}_{J_6},
\end{split}
\end{equation*}
where the following fact
$$
-\left((\pa_x\widetilde{n}_\al)^2, n_\al^{-3}\pa_tn_\al\right)
-\left(\pa_x u_\al\pa_x\widetilde{n}_\al+u_\al\pa^2_x\widetilde{n}_\al, \frac{\pa_x\widetilde{n}_\al}{n_\al^2}\right)=
\frac{1}{2}\left(\pa_x u_\al\pa_x\widetilde{n}_\al,\frac{\pa_x\widetilde{n}_\al}{n_\al^2}\right)
$$
has been used to deduce $J_{2}$. It is straightforward to see that
all terms $J_l$ for $1\leq l\leq 6$ can be estimated as for $I_{l}$ for $10\leq l\leq 15$ on the right hand side of \eqref{ztrho.sum1}.
Then, by applying those estimates to \eqref{2ztrho.sum1}  and letting $0<\eta<1$ and $\eps_0>0$ be suitably small, we obtain
\begin{eqnarray}
&&\frac{d}{dt} \sum\limits_{\al=i,e}m_\al(\widetilde{u}_\al, \pa_x\widetilde{n}_\al)
+\frac{1}{2}\frac{d}{dt} \sum\limits_{\al=i,e}((\pa_x\widetilde{n}_\al)^2, n_\al^{-2})+\sum_{\al=i,e}\left(\pa_xu^r, (\pa_x\widetilde{n}_\al)^2n_\al^{-2}\right)\notag\\
&&\quad +\la\sum\limits_{\al=i,e}\left(n_\al^{-1}, (\pa_x\widetilde{n}_\al)^2\right)
+\la(\pa^2_x\widetilde{\phi},\pa^2_x\widetilde{\phi})\notag\\
&&\leq \max\{1,n_+\}\sum\limits_{\al=i,e}m_\al\|\pa_x\widetilde{u}_\al\|^2 +C\eps_0^2\sum\limits_{\al=i,e}\|\pa_x^2\widetilde{u}_\al\|^2\notag\\
&&\quad+ C(1+t)^{-2} \sum\limits_{\al=i,e}\|[\widetilde{n}_\al,\widetilde{u}_\al]\|^2+C\eps (1+t)^{-2}.\label{2s2.f}
\end{eqnarray}
Further taking the summation of $2\times$\eqref{2s2.f} and ${2C_2}\times$\eqref{2zeui+ue3} for a suitably large constant ${C_2}>1$ and also letting $\eps_0>0$ and $0<\eta<1$ in \eqref{2ztrho.sum1} be suitably small gives
\begin{equation}\label{2ztru.sum1}
\begin{split}
\frac{d}{dt} &\sum\limits_{\al=i,e}\left\{2m_\al(\widetilde{u}_\al, \pa_x\widetilde{n}_\al)+((\pa_x\widetilde{n}_\al)^2, n_\al^{-2})\right\}
+\frac{d}{dt} {C_2}\sum\limits_{\al=i,e}\left\{m_\al(\widetilde{u}_\al, n_\al\widetilde{u}_\al)
+2(n_\al, \psi_\al)\right\}\\&
+{C_2}\frac{d}{dt}(\pa_x\widetilde{\phi},\pa_x\widetilde{\phi})
-\frac{C_2\be m_i}{\be+\ga}\frac{d}{dt}(\widetilde{u}_i,\pa_x\widetilde{\phi}\pa_xu^r)
+\frac{C_2\ga m_e}{\be+\ga}\frac{d}{dt}(\widetilde{u}_e,\pa_x\widetilde{\phi}\pa_xu^r)\\
&
+\la(\pa^2_x\widetilde{\phi},\pa^2_x\widetilde{\phi})+\la\sum\limits_{\al=i,e}\left\{m_\al(\widetilde{u}_\al^2,n_\al\pa_xu^r)+\left(\pa_xu^r, (\pa_x\widetilde{n}_\al)^2n_\al^{-2}\right)\right\}
\\&+\la\sum\limits_{\al=i,e}\left\{(\pa_x\widetilde{u}_\al,\pa_x\widetilde{u}_\al)+\left(n_\al^{-1}, (\pa_x\widetilde{n}_\al)^2\right)\right\}
\\
\lesssim & \eps_0\sum\limits_{\al=i,e}\|\pa_x^2\widetilde{u}_\al\|^{2}
+(1+t)^{-2}\sum\limits_{\al=i,e}\|[\widetilde{n}_\al,\widetilde{u}_\al,\pa_x\widetilde{\phi}]\|^2
+(1+t)^{-\frac{3+5\sigma}{2(1+3\sigma)}}\sum\limits_{\al=i,e}\|\widetilde{u}_\al\|^{2}\\
&+\eps^{\frac{\sigma}{1+\sigma}}(1+t)^{-\frac{5+3\sigma}{4+4\sigma}}
+\eps(1+t)^{-2}.
\end{split}
\end{equation}
Here, as $\|\pa_x u^r\|_{L^\infty}\leq C\delta_r\eps$, in terms of  Cauchy-Schwarz's inequality, one can take $C_2>0$ large enough and $\eps>0$ small enough such that
\begin{multline}\label{re-ad1}
2m_\al(\widetilde{u}_\al, \pa_x\widetilde{n}_\al)+((\pa_x\widetilde{n}_\al)^2, n_\al^{-2})
+{C_2}m_\al(\widetilde{u}_\al, n_\al\widetilde{u}_\al)
+{C_2}(\pa_x\widetilde{\phi},\pa_x\widetilde{\phi})\\
-\frac{C_2\be m_i}{\be+\ga}(\widetilde{u}_i,\pa_x\widetilde{\phi}\pa_xu^r)
+\frac{C_2\ga m_e}{\be+\ga}(\widetilde{u}_e,\pa_x\widetilde{\phi}\pa_xu^r)
\thicksim \|\pa_x\widetilde{n}_\al\|^2+\|\widetilde{u}_\al\|^2+\|\pa_x\widetilde{\phi}\|^2.
\end{multline}
{By applying \eqref{2ztru.sum1} and Gronwall's inequality, we conclude}
\begin{equation}\label{2ztru.sum2}
\begin{split}
\sup\limits_{0\leq t\leq T} &\Big\{\sum\limits_{\al=i,e}\left\{(2m_\al(\widetilde{u}_\al, \pa_x\widetilde{n}_\al)+((\pa_x\widetilde{n}_\al)^2, n_\al^{-2})+{C_2}m_\al(\widetilde{u}_\al, n_\al\widetilde{u}_\al)
+2{C_2}(n_\al, \psi_\al)\right\}\\&\qquad\qquad+{C_2}(\pa_x\widetilde{\phi}, \pa_x\widetilde{\phi})
-\frac{C_2\be m_i}{\be+\ga}(\widetilde{u}_i,\pa_x\widetilde{\phi}\pa_xu^r)
+\frac{C_2\ga m_e}{\be+\ga}(\widetilde{u}_e,\pa_x\widetilde{\phi}\pa_xu^r)
\Big\}\\
&+\la\int_0^T\left\{\sum\limits_{\al=i,e}\left\{m_\al(\widetilde{u}_\al^2,n_\al\pa_xu^r)+\left(\pa_xu^r, (\pa_x\widetilde{n}_\al)^2n_\al^{-2}\right)\right\}\right\}dt
\\&+\la\int_0^T\left\{\sum\limits_{\al=i,e}\left\{(\pa_x\widetilde{u}_\al,\pa_x\widetilde{u}_\al)+\left(n_\al^{-1}, (\pa_x\widetilde{n}_\al)^2\right)\right\}+(\pa^2_x\widetilde{\phi},\pa^2_x\widetilde{\phi})\right\}dt
\\
\lesssim & \sum\limits_{\al=i,e}\left\{\|[\widetilde{n}_\al,\widetilde{u}_\al](0)\|^2+\|\pa_x \widetilde{n}_\al(0)\|^2\right\}
+\|\pa_x \widetilde{\phi}(0)\|^2+\eps_0\sum\limits_{\al=i,e}\int_0^T\|\pa^2_x\widetilde{u}_\al\|^2dt+\eps^{\frac{\sigma}{1+\sigma}}.
\end{split}
\end{equation}

\noindent{\bf Step 3.} {\it Dissipation of $\pa_x^2\widetilde{u}_\al$.}

By differentiating \eqref{2tui} and \eqref{2tue} in $x$, taking the inner products of the resulting identities with $\pa_x\widetilde{u}_i$
and $\pa_x\widetilde{u}_e$
with  respect to $x$ over $\R$, respectively, and taking the summation for $\al=i$ and $e$, we have
\begin{equation}\label{2d.ui-ue}
\begin{split}
\sum\limits_{\al=i,e}m_\al&(\pa_t \pa_x\widetilde{u}_\al,\pa_x\widetilde{u}_\al)+\sum\limits_{\al=i,e}m_\al(\pa_x(\widetilde{u}_\al\pa_x \widetilde{u}_\al+\widetilde{u}_\al\pa_x u^r+\pa_x \widetilde{u}_\al u^r), \pa_x\widetilde{u}_\al)\\&+\sum\limits_{\al=i,e}T_\al\left(\pa_x\left(\frac{\pa_xn_\al}{n_\al}-\frac{\pa_xn^r}{n^r}\right), \pa_x\widetilde{u}_\al\right)
+(\pa^2_x\widetilde{\phi}, -\pa_x\widetilde{u}_i+\pa_x\widetilde{u}_e)\\&+\sum\limits_{\al=i,e}\left(-\pa_x\left(\frac{\pa_x^2\widetilde{u}_\al}{n_\al}\right),\pa_x\widetilde{u}_\al\right)
+\sum\limits_{\al=i,e}\left(-\pa_x\left(\frac{\pa_x^2u^r}{n_\al}\right),\pa_x\widetilde{u}_\al\right)=0.
\end{split}
\end{equation}
It is straightforward to see that all terms in \eqref{2d.ui-ue} can be estimated in the completely same way as in \eqref{tu.uip}, so that we directly arrive at
\begin{multline}\label{2tu.uip.p1}
\frac{1}{2}\frac{d}{dt}\sum\limits_{\al=i,e}m_\al(\pa_x\widetilde{u}_\al,\pa_x\widetilde{u}_\al)+\frac{3}{2}\sum\limits_{\al=i,e}m_\al(\pa_x u^r, (\pa_x\widetilde{u}_\al)^2)
+\la\sum\limits_{\al=i,e}\left(\frac{\pa_x^2\widetilde{u}_\al}{n_\al},\pa^2_x\widetilde{u}_\al\right)\\
\leq
C\eps(1+t)^{-2}+C\left(\sum\limits_{\al=i,e}\|\pa_x [\widetilde{n}_\al,\pa_x \widetilde{u}_\al]\|^2+\|\pa^2_x\widetilde{\phi}\|^2\right).
\end{multline}
Integrating \eqref{2tu.uip.p1} with respect to $t$ over $[0,T]$ and letting $0<\eta<1$ be suitably small, one further has
\begin{equation}\label{2tu.uip.p2}
\begin{split}
\sup\limits_{0\leq t\leq T}&\sum\limits_{\al=i,e}m_\al(\pa_x\widetilde{u}_\al,\pa_x\widetilde{u}_\al)
+\sum\limits_{\al=i,e}m_\al\int_0^T(\pa_x u^r, (\pa_x\widetilde{u}_\al)^2)dt
\\&+\la\sum\limits_{\al=i,e}\int_0^T\left(\frac{\pa_x^2\widetilde{u}_\al}{n_\al},\pa^2_x\widetilde{u}_\al\right)dt\\
\lesssim &\sum\limits_{\al=i,e}m_\al\|\pa_x \widetilde{u}_\al(0)\|^2
+C\sum\limits_{\al=i,e}\int_0^T(\|\pa_x [\widetilde{n}_\al,\pa_x \widetilde{u}_\al]\|^2+\|\pa^2_x\widetilde{\phi}\|^2)dt
+C\eps.
\end{split}
\end{equation}
Thus {we get from \eqref{2ztru.sum2} and \eqref{2tu.uip.p2} that}
\begin{equation}\label{2total.eng2.}
\begin{split}
\sup\limits_{0\leq t\leq T} &\Big\{\sum\limits_{\al=i,e}\left\{2m_\al(\widetilde{u}_\al, \pa_x\widetilde{n}_\al)+((\pa_x\widetilde{n}_\al)^2, n_\al^{-2})+{C_2}m_\al(\widetilde{u}_\al, n_\al\widetilde{u}_\al)
+2{C_2}(n_\al, \psi_\al)\right\}\\&\quad+\sum\limits_{\al=i,e}m_\al(\pa_x\widetilde{u}_\al,\pa_x\widetilde{u}_\al)+{C_2}(\pa_x\widetilde{\phi}, \pa_x\widetilde{\phi})
-\frac{C_2\be m_i}{\be+\ga}(\widetilde{u}_i,\pa_x\widetilde{\phi}\pa_xu^r)
+\frac{C_2\ga m_e}{\be+\ga}(\widetilde{u}_e,\pa_x\widetilde{\phi}\pa_xu^r)
\Big\}
\\&+\la\int_0^T\left\{\sum\limits_{\al=i,e}\left\{m_\al(\widetilde{u}_\al^2,n_\al\pa_xu^r)+\left(\pa_xu^r, (\pa_x\widetilde{n}_\al)^2n_\al^{-2}\right)\right\}\right\}dt
\\&+\la\int_0^T\left\{\sum\limits_{\al=i,e}\left\{(\pa_x\widetilde{u}_\al,\pa_x\widetilde{u}_\al)+\left(n_\al^{-1}, (\pa_x\widetilde{n}_\al)^2\right)\right\}+(\pa^2_x\widetilde{\phi},\pa^2_x\widetilde{\phi})\right\}dt
\\&+\la\int_0^T\left\{\sum\limits_{\al=i,e}\left\{m_\al((\pa_x\widetilde{u}_\al)^2,n_\al\pa_xu^r)+
\left(\frac{\pa_x^2\widetilde{u}_\al}{n_\al},\pa^2_x\widetilde{u}_\al\right)\right\}\right\}dt
\\
\lesssim & \sum\limits_{\al=i,e}\left\{\|[\widetilde{n}_\al,\widetilde{u}_\al](0)\|_{H^1}^2\right\}
+\|\pa_x \widetilde{\phi}(0)\|^2+\eps^{\frac{\sigma}{1+\sigma}}.
\end{split}
\end{equation}
Recalling \eqref{re-ad1}, notice that
\begin{multline*}
\sum\limits_{\al=i,e}\left\{(2m_\al\widetilde{u}_\al, \pa_x\widetilde{n}_\al)+((\pa_x\widetilde{n}_\al)^2, n_\al^{-2})+{C_2}m_\al(\widetilde{u}_\al, n_\al\widetilde{u}_\al)
+2{C_2}(n_\al, \psi_\al)\right\}\\
+\sum\limits_{\al=i,e}m_\al(\pa_x\widetilde{u}_\al,\pa_x\widetilde{u}_\al)+{C_2}(\pa_x\widetilde{\phi}, \pa_x\widetilde{\phi})
-\frac{C_2\be m_i}{\be+\ga}(\widetilde{u}_i,\pa_x\widetilde{\phi}\pa_xu^r)
+\frac{C_2\ga m_e}{\be+\ga}(\widetilde{u}_e,\pa_x\widetilde{\phi}\pa_xu^r)\\
\thicksim \sum_{\al=i,e}\left(\left\|\widetilde{n}_\al\right\|^2_{H^1}+\|\widetilde{u}_\al\|_{H^1}^2\right)+\|\pa_x\widetilde{\phi}\|^2,
\end{multline*}
and
\begin{equation*}
\begin{split}
(\pa_x\widetilde{u}_\al,\pa_x\widetilde{u}_\al)+\left(n_\al^{-1}, (\pa_x\widetilde{n}_\al)^2\right)+\left(\frac{\pa_x^2\widetilde{u}_\al}{n_\al},\pa^2_x\widetilde{u}_\al\right)
\thicksim& \left\|[\pa_x\widetilde{n}_\al,\pa_x\widetilde{u}_\al,\pa^2_x\widetilde{u}_\al]\right\|^2,
\end{split}
\end{equation*}
according to the a priori assumption \eqref{2pri.ass.} and the fact that $n_+>n_->0$.
{\eqref{2total.eng1.} then follows from the above observation and \eqref{2total.eng2.}}. Notice that the boundedness of $\|\pa_x^2\widetilde{\phi}(t)\|^2$ for all $t$ directly follows from the Poisson equation \eqref{2tphy}. This then completes the proof of Proposition \ref{2a.pri.}.
\end{proof}

We are now in a position to complete the

\begin{proof}[Proof of Theorem \ref{2main.res.}.]
The existence of the solution follows from the standard continuity argument based on the local existence and the a priori
estimate in Proposition \ref{2a.pri.}. Then \eqref{2total.eng.} holds true. The large-time behavior given in \eqref{2asp} and \eqref{2asp-ad} can be verified in terms of \eqref{2total.eng.}. This ends the proof of Theorem \ref{2main.res.}.
\end{proof}

\section{Appendix}

In this appendix, we present the dissipative structure (cf.~\cite{UDK} and references therein) of the linearized system corresponding to the one-fluid model \eqref{NSP} around a constant equilibrium state $[n,u]=[1,0]$ with $\phi=0$, namely
\begin{eqnarray*}
&&\pa_t n +\pa_x u=0,\\
&&\pa_t u +A \pa_x n -\pa_x \phi =\pa_x^2 u,\\
&&{\varepsilon} \pa_x^2\phi =n +\phi,
\end{eqnarray*}
where we put a constant ${\varepsilon}>0$ in front of $\pa_x^2\phi$ in order to see what happens as ${\varepsilon}$ tends to zero. The case of two-fluid model \eqref{NSP2} can be considered in a similar way; see \cite{HL} for details.  In fact, taking  the Fourier transform in $x\in \R$ gives
\begin{eqnarray*}
&&\pa_t \widehat{n}+i\xi \widehat{u}=0,\\
&&\pa_t \widehat{u} +A i\xi \widehat{n}-i\xi \widehat{\phi}=-\xi^2 \widehat{u},\\
&&-a(\xi)\widehat{\phi}=\widehat{n},
\end{eqnarray*}
where $a(\xi)={\varepsilon} \xi^2+1$. The direct energy estimate implies
\begin{equation}
\label{ap.1}
\pa_t \left[\left(A+\frac{1}{a(\xi)}\right) |\widehat{n}|^2+|\widehat{u}|^2\right]+2\xi^2 |\widehat{u}|^2=0.
\end{equation}
Moreover, it is also straightforward to compute
\begin{equation*}
\pa_t \lag i\xi \widehat{n},\widehat{u} \rag +(A+\frac{1}{a(\xi)})\xi^2 |\widehat{n}|^2 =\xi^2 |\widehat{u}|^2+\lag i \xi \widehat{n}, -\xi^2 \widehat{u}\rag,
\end{equation*}
which after taking the real part of the equation, applying the Cauchy-Schwarz inequality to the last term and dividing it by $1+\xi^2$, leads to
\begin{equation}
\label{ap.2}
\pa_t \frac{\Re \lag i\xi \widehat{n},\widehat{u} \rag}{1+\xi^2} +\frac{\xi^2}{2(1+\xi^2)}  (A+\frac{1}{a(\xi)}) |\widehat{n}|^2\leq C \xi^2 |\widehat{u}|^2.
\end{equation}
Here $\lag\,,\,\rag$ stands for the complex inner product. Then, it follows from \eqref{ap.1} and \eqref{ap.2} that
\begin{equation}
\label{ap.3}
\pa_t E +\la D\leq 0,
\end{equation}
with
\begin{eqnarray*}
E & =& \left(A+\frac{1}{a(\xi)}\right) |\widehat{n}|^2+|\widehat{u}|^2 +\ka   \frac{\Re \lag i\xi \widehat{n},\widehat{u} \rag}{1+\xi^2},\\
D & = & \xi^2 |\widehat{u}|^2+\frac{\xi^2}{1+\xi^2}  (A+\frac{1}{a(\xi)}) |\widehat{n}|^2,
\end{eqnarray*}
where $\ka>0$ is suitably small. Noticing
\begin{equation*}
E \sim \left(A+\frac{1}{a(\xi)}\right) |\widehat{n}|^2+|\widehat{u}|^2,
\end{equation*}
it further follows from \eqref{ap.3} that
\begin{equation*}
\pa_t E +\frac{\la \xi^2}{1+\xi^2} E \leq 0.
\end{equation*}
Therefore, $E(t,\xi)\leq e^{-\frac{\la \xi^2}{1+\xi^2} t}E(0,\xi)$ holds true, and this directly implies the time-decay property of the linearized solution operator as in \cite{UDK}. Notice that similar to obtain \eqref{ap.3} in the Fourier space,  it also holds in the original space that
\begin{equation*}
\frac{d}{dt} \left( \|(\sqrt{A}n,u)\|_{H^1}^2+\|\phi\|_{H^2}^2 +\ka (u,\pa_x n)\right)
+\la (\|\pa_x n\|^2+\|\pa_x u\|_{H^1}^2 +\|\pa_x\phi\|_{H^2}^2)\leq 0.
\end{equation*}
Finally, we also write down the Green's matrix $G(t,\xi)$ corresponding to the linearized system:
\begin{equation*}
G(t,\xi)=e^{\la_+ t} P_+ +e^{\la_-t} P_-,
\end{equation*}
where
\begin{equation*}
\la_\pm=\frac{-\xi^2\pm \sqrt{\xi^4-4\xi^2 a(\xi)}}{2}
\end{equation*}
are the eigenvalues of the coefficient matrix
\begin{equation*}
M=\begin{pmatrix}
     0 &\,-i\xi  \,  \\[1mm]
    -i\xi a(\xi)  &\, -\xi^2 \,
\end{pmatrix},
\end{equation*}
and
\begin{equation*}
P_\pm=\frac{M-\la_\mp I}{\la_\pm-\la_\mp}
\end{equation*}
are the corresponding eigenprojection,  with $I$ being a $2\times2$ identity matrix. Therefore, it is direct to obtain
\begin{equation*}
G(t,\xi)=\begin{pmatrix}
 \ \dis \frac{\la_+ e^{\la_- t} -\la_- e^{\la_+ t}}{\la_+ -\la_-}  \  &\  \dis-i \xi \frac{e^{\la_+ t} -e^{\la_-t}}{\la_+-\la_-}  \ \\[3mm]
\ \dis-i \xi a(\xi) \frac{e^{\la_+ t} -e^{\la_-t}}{\la_+-\la_-}    \   &\   \dis \frac{\la_+ e^{\la_+ t} -\la_- e^{\la_- t}}{\la_+ -\la_-}\
\end{pmatrix}.
\end{equation*}
Observe that $G(t,\xi)$ must reduce to the Green's matrix in the Fourier space for the one-dimensional Navier-Stokes equations as ${\varepsilon}\to 0$, cf.~\cite{HZ2}, for instance.

\medskip
\noindent {\bf Acknowledgements:} RJD was supported by the General Research Fund (Project No.~400912) from RGC of Hong Kong. SQL was
supported by grants from the National Natural Science Foundation of China under contracts 11471142 and 11271160. The authors would like to thank Xiongfeng Yang for many fruitful discussions on the two-fluid case, and also thank the anonymous referee for very helpful comments on the paper.


\end{document}